\documentclass[11pt]{amsart}
\pdfoutput=1

\usepackage{amsmath}%\
\usepackage{amsfonts}%
\usepackage{amssymb}%
\usepackage{amsthm}
\usepackage{xcolor}
\usepackage{listings}
\definecolor{mygreen}{rgb}{0,0.6,0}
\definecolor{mygray}{rgb}{0.5,0.5,0.5}
\definecolor{mymauve}{rgb}{0.58,0,0.82}
\definecolor{backcolour}{rgb}{0.95,0.95,0.92}

\usepackage{xspace}

\usepackage{graphicx}

\usepackage{float}

\usepackage{multirow}% http://ctan.org/pkg/multirow
\usepackage{graphicx, import}
\usepackage{tikz-cd}
\usepackage{hyperref}
\usepackage{cleveref}
\usepackage{a4wide}

\usepackage{subcaption}

\usepackage{parskip}

\crefrangelabelformat{equation}%Writing multiple references as (1a-1c)
{(#3#1#4--#5\crefstripprefix{#1}{#2}#6)}
\crefrangelabelformat{subequation}%
{(#3#1#4--#5\crefstripprefix{#1}{#2}#6)}

\newtheorem{theorem}{Theorem}[section]
\newtheorem*{mainthm*}{\hypertarget{thm:DRclosure}{\bf Main theorem}}
\newtheorem{theorem+definition}{Theorem\,+\,Definition}[theorem]

\numberwithin{case}{theorem}
\newtheorem{claim}[theorem]{Claim}
\newtheorem*{claim*}{Claim}

\newtheorem{corollary}[theorem]{Corollary}

\newtheorem{question}[theorem]{Question}

\newtheorem{lemma}[theorem]{Lemma}

\newtheorem{convention}[theorem]{Convention}

\newtheorem{proposition}[theorem]{Proposition}

\newtheorem{def+prop}[theorem]{Definition\,+\,Proposition}
\theoremstyle{definition}
\newtheorem{definition}[theorem]{Definition}
\theoremstyle{remark}
\newtheorem{remark}[theorem]{Remark}

\DeclareMathOperator{\BIC}{BIC}

\DeclareMathOperator{\spc}{sp}
\DeclareMathOperator{\GL}{GL}

\DeclareMathOperator{\CC}{\mathbb{C}}
\DeclareMathOperator{\RR}{\mathbb{R}}
\DeclareMathOperator{\PP}{\mathbb{P}}
\DeclareMathOperator{\ZZ}{\mathbb{Z}}

\DeclareMathOperator{\rank}{rank}
\DeclareMathOperator{\Res}{Res}

\DeclareMathOperator{\mult}{mult}

\DeclareMathOperator{\id}{id}

\DeclareMathOperator{\Bl}{Bl}

\DeclareMathOperator{\res}{res}
\DeclareMathOperator{\ord}{ord}

\newcommand{\codim}{\operatorname{codim}}

\DeclareMathOperator{\ANN}{Ann}

\DeclareMathOperator{\abs}{abs}

\DeclareMathOperator{\Spec}{\operatorname{Spec}}

\newcommand\calV{\mathcal{V}}

\newcommand\calU{\mathcal{U}}

\newcommand\calO{\mathcal{O}}
\newcommand\calN{\mathcal{N}}

\newcommand\calE{\mathcal{E}}
\newcommand\calR{\mathcal{R}}
\newcommand\calI{\mathcal{I}}

\newcommand\calM{\mathcal{M}}

\newcommand\calH{\mathcal{H}}

\newcommand\calS{\mathcal{S}}
\newcommand\calP{\mathcal{P}}

\newcommand\calJ{\mathcal{J}}

\newcommand\Adm[1][\ramp]{\operatorname{Adm}_{#1} }
\newcommand\Hur[1][\ramp]{\operatorname{Hur}_{#1} }

\newcommand\Hurcl[1][\ramp]{\overline{\Hur[]}_{#1} }

\newcommand{\Mgn}[1][g,n]{\calM_{#1}}

\newcommand{\Mgnbar}[1][g,n]{ \overline\calM_{#1}}

\newcommand\rescl[1][i]{T_{\lfloor #1\rfloor}}

\newcommand{\cl}[1]{\overline{#1}}

\newcommand\DM{Deligne-Mumford\xspace}

\newcommand{\bd}[1][\pf]{D_{#1}}
\newcommand{\bdint}[1][\pf]{D^{\circ}_{#1}}
\newcommand\lG{\overline{\Gamma}}%enhancwed level graph
\newcommand\lGh{\overline{\Lambda}}

\newcommand\sG{\Gamma}%stable graph

\newcommand\lexp[1][i]{\ell_{#1}}

\newcommand\LG{\operatorname{LG}}
\newcommand\LGc[1][1]{\LG_1}

\newcommand\CH[1][*]{\operatorname{CH}^{#1}}
\newcommand\logCH[1][*]{\operatorname{logCH}^{#1}}
\newcommand\CHop[1][*]{\operatorname{CH}^{#1}_{\operatorname{op}}}
\newcommand\ch{\operatorname{ch}}

\newcommand\Rtr[1][*]{R^{#1}}

\newcommand\St[1][(\mu)]{\calH{#1}}
\newcommand\PSt[1][(\mu)]{\calP #1}
\newcommand{\PMSDS}[1][(\mu)]{\overline{\calP}#1}

\newcommand{\PMSDSdisc}[2]{\PP\Xi\Mgnbar[#1] #2}

\newcommand\taut{\xi}

\newcommand\pf[1][P]{\operatorname{#1}}
\newcommand\pfQ{\pf[Q]}
\newcommand\pfR{\pf[R]}
\newcommand\pfS{\pf[S]}
\newcommand\pfP{\pf[P]}

\newcommand\PF[1][(B)]{\Sigma #1}
\newcommand\PFtop[1][{(\pf[Q])}]{\mathcal{P}^{\top}#1}
\newcommand\PFi[2][(B)]{\Sigma^{#2}#1}

\DeclareMathOperator{\ev}{ev}

\newcommand\vb{\calV}

\tikzset{
  symbol/.style={
    draw=none,
    every to/.append style={
      edge node={node [sloped, allow upside down, auto=false]{$#1$}}}
  }
}

\newlength{\mylen}
\setbox1=\hbox{$\bullet$}\setbox2=\hbox{\tiny$\bullet$}
\newcommand\undeg[1][{[i]}]{\overset{{#1}}{\leadsto}} 
\newcommand\degen{{\leadsto}}

\newcommand\coht[1][{[i]}]{V^{#1}_{\ramp}(X,\omega)}

\newcommand\ramp{\underline{\lambda}}

\newcommand\chp[3][{}]{P_{{#2}/{#3}}#1}%Chern polynomial
\newcommand\CHP{total Chern class\xspace}
\newcommand\CHPs{total Chern classes\xspace}

\newcommand\Ann[2][(\vb)]{\ANN_{#2} #1}
\newcommand\clAnn[2][(\vb)]{\cl{\ANN_{#2}#1}}

\newcommand\nb[2]{\calN_{#1/#2}}
\newcommand\ci[2][i]{c_{#1}(#2)}
\newcommand\cinb[3][i]{\ci[#1]{\nb{#2}{#3}}}
\newcommand\cinbe[3][i]{\overline{c}_{#1}(\nb{#2}{#3})}

\newcommand\divR[1][*]{\operatorname{divR}^{#1}}
\newcommand\tr[2][*]{\operatorname{R}^{#1}(#2)}
\newcommand\trd[2][*]{\operatorname{divR}^{#1}(#2)}

\newcommand\bdexl[2][\pfP,\pfQ]{A_{#1}^{[#2]}}

\newcommand\Exc[1][\ramp]{\operatorname{Exc}_{#1}}
\newcommand\Exccl[1][\ramp]{\overline{\operatorname{Exc}}_{#1}}

\newcommand\Excmu[1][(\mu)]{\operatorname{Exc}{#1}}
\newcommand\Excclmu[1][(\mu)]{\overline{\operatorname{Exc}}{#1}}

\newcommand\ptf[1]{\widetilde{#1}}% proper transform

\numberwithin{equation}{section}

\newlength{\perspective}

\title{Tautological classes of strata of exact differentials}
\author{Frederik Benirschke}
\begin{document}

\begin{abstract}
Strata of exact differentials are moduli spaces for differentials on Riemann surfaces with vanishing absolute periods.
Our main result is that classes of closures of strata of exact differentials inside the moduli space of multi-scale differentials lie in the divisorial tautological ring. By relating exact differentials to rational functions  we obtain a new proof that classes of Hurwitz spaces are tautological and a new method for computations.
%For low degree covers we obtain explicit formulas in terms of tautological classes for the cycles classes inside $\Mgnbar$.

\end{abstract}
\maketitle
\setcounter{tocdepth}{1}
\tableofcontents
\section{Introduction}
\subsection{Motivation from algebraic geometry: Cycle classes of Hurwitz spaces}
A classical way of producing subvarieties of the moduli space of curves is using Hurwitz spaces.
The moduli space $\Hur$ of Hurwitz covers of $\PP^1$ parametrizes finite degree morphisms $f:X\to\PP^1$ with prescribed ramification.
A {\em ramification profile} for degree $d$ covers is a tuple $\ramp=
(\lambda_1,\ldots,\lambda_m)$ of ordered partitions of $d$.
The Hurwitz space $\Hur$ parametrizes branched covers $f:X\to\PP^1$ of degree $d$ branched over $m$ points $x_1,\ldots,x_m$ in $\PP^1$ with ramification described by $\lambda_i$ in the fiber over $x_i$, and no other branch points.
A modular compactification  $\Hurcl$ of Hurwitz spaces using admissible covers was first constructed by Harris and Mumford.
%, and later clarified in a general setting in [ref].
Forgetting the branched cover yields a morphism $\phi:\Hurcl\to\Mgnbar$ and pushing forward the fundamental class  produces natural {\em admissible cover cycles} 
$\Adm\in\CH(\Mgnbar)$.
The computation of admissible cover cycles has attracted a lot of research, see \cite{eisenbud-harris,faber-pandharipande,faber-pagani,schmitt-vanzelm} for a collection of results.
While in general no closed formulas for admissible cover cycles are known, it was shown in \cite{faber-pandharipande} that admissible cover cycles lie in the tautological ring $\tr{\Mgnbar}$.

In order to study admissible cover cycles we will now change our point of view and work with differential forms instead of rational functions. This perspective has appeared already in \cite{Fred,BDG}, where we constructed a smooth compactification of Hurwitz spaces using exact differentials.
Similar ideas were also used in \cite{mullane,sauvaget-thesis}.

Given a rational function $f:X\to \PP^1$ from a Riemann surface $X$ of genus $g$, the {\em associated exact differential} is  $df=f^*dz$ and we can recover $f$, up to additive constants, by integrating. 
The key observation is that exact differentials can be characterized by the vanishing of all absolute periods
\begin{equation}\label{eq:abper}
\int_{\gamma} df=0 \text{ for all } \gamma \in H_1(X\setminus P(df);\CC),
\end{equation}
where $P(df)\subseteq X$ is the set of poles of $df$.
\Cref{eq:abper} can be rephrased as lying in the zero locus of a section of a vector bundle over a suitable moduli space of differentials, which will allow us to perform intersection theory. We will now setup the moduli spaces in question.

\subsection{Strata of exact differentials}
The stratum $\St$ parametrizes tuples \[
(X,Q=(q_1,\ldots,q_n),\omega)\] of  marked Riemann surfaces $X$ together with a meromorphic differential $\omega$ vanishing to  order $\mu_i$ at $q_i$. Most of the time we only consider differentials up to rescaling, which are parametrized by the projectivized stratum $\PSt=\St/\CC^*$.

Given a ramification profile $\ramp$ we let $\mu$ be the partition recording the zero and pole orders of $df$.
For example if we consider  hyperelliptic curves the ramification profile is $\ramp=(2^{2g+2})$ and the associated partition is $\mu=(-3,1^{2g+1})$, where we used exponential notation to denote  repetitions.

The {\em stratum of  exact differentials}
\[
\begin{split}
\Exc[](\mu)&
=\left\{(X,\omega)\,:\, \int_{\gamma}\omega=0 \text{ for all }\gamma \in H_1(X\setminus P(\omega);\ZZ)\right\}\subseteq \PSt
\end{split}
\]
is either empty or a codimension $2g+|P(\omega)|-1$ subvariety of the projectivized stratum.

Differentials in $\Excmu$ correspond to rational functions with prescribed ramification multiplicities, but not necessarily with ramification profile $\ramp$, since so far we have no requirement on which ramification points lie in the same fiber. In \Cref{sec:prelim} we will construct the {\em stratum of $\ramp$-exact differentials} $\Exc\subset\Excmu$ consisting of exact differentials coming from rational functions with ramification profile $\ramp$.
In order to perform intersection theory  we will work on a compactification of the stratum $\PSt$.

Recently  in \cite{BCGGMsm} the authors construct the {\em moduli space of multi-scale differentials} $\PMSDS$, a smooth compactification of the stratum $\PSt$. We will recall its main features in \Cref{sec:localsystems}.
The boundary of $\PMSDS$ is a normal crossing divisor, which can be decomposed into the divisor $D_h$ consisting of irreducible curves with horizontal nodes and divisors $D_{\lG}$ parametrized by non-horizontal two level graphs. Similar to $\Mgnbar$, the moduli space of multi-scale differentials has a natural tautological ring $\tr{\PMSDS}$, which was first defined in \cite{CMZ}.
For example the ring $\tr{\PMSDS}$ contains all boundary divisors $\bd[\lG]$ for non-horizontal two-level graphs and also the first Chern class 
\[
\taut{}{}=c_1(\calO(-1))\in\CH[1](\PMSDS)
\] of the tautological line bundle, whose fiber over a twisted differential is the line generated by the differential.

Our main result is the computation of the class of the closure $\Excclmu$ in the Chow ring of $\PMSDS$.

\begin{theorem}\label{thm:main}Let $\mu$ be a partition of $2g-2$.
%and $d=2g-2+|P(\omega)|-1$.
The class of the stratum of exact differentials $\Excclmu$ lies in the divisorial tautological ring $\trd{\PMSDS}\subseteq \CH(\PMSDS)$ generated by $\taut$ and $\bd[\lG]$ where $\lG$ is a non-horizontal boundary divisor.
Furthermore, there exists an explicit algorithm computing the class  $\Excclmu$.
The same is true for the class of $\ramp$-exact strata $\cl{\Exc}$ for any ramification profile $\ramp$.

\end{theorem} 
In particular the class of $\Excclmu$ behaves like the class of a complete intersection, although we do not expect $\Excclmu$ to be a complete intersection unless $g=0$.

We now circle back to the connection between rational functions and exact differentials.
The moduli space $\PMSDS$ comes with a forgetful map $\rho:\PMSDS\to \Mgnbar$ and by construction the pushforward of 
$[\Exccl]$ is the admissible cover cycle $\Adm$.

\begin{corollary}[\cite{faber-pandharipande}]\label{cor:taut}
 Admissible cover cycles $\Adm\in\CH(\Mgnbar)$ are tautological.
\end{corollary}

\begin{proof}
The pushforward $\rho_*:\CH(\PMSDS)\to\CH(\Mgnbar)$ sends tautological classes to tautological classes (see \Cref{prop:taut}) and the statement now follows directly from \Cref{thm:main} for strata of $\ramp$-exact differentials.
\end{proof}

The proof of \Cref{thm:main} is explicit and together with the algorithm for the computation of the pushforward $\rho_*$ in \cite{CMZ}, we obtain an explicit approach for calculating admissible cover cycles.

The algorithm for the computation of strata of $\ramp$-exact differentials has been implemented in \verb|sage| using the packages \verb|admcycles| \cite{admcycles} and \verb|diffstrata| \cite{diffstrata} and can be used to cross check our formulas, see \Cref{section-crosschecks} for examples.

\subsection{Closed formulas in low genus}
While in general our formula only produces a recursive algorithm to compute the class of strata of exact differentials, in  genus $g=0,1$ we can use \Cref{thm:main}  to extract closed formulas for the class of the stratum of $\ramp$-exact differentials.

\begin{proposition} Let $\Exc$ be a stratum $\ramp$-exact differentials in $g=0,1$. Then
\[
[\cl{\Exc}] = (-1)^{\codim \Exc} \prod_{i} (\xi + \sum_{\lG\in W_i(\ramp)} \ell_{\lG}\bd[\lG]) \cdot \alpha_{\ramp} \in \CH(\PMSDS)
\] for some explicit collection of two-level graphs $W_i(\ramp)$
 and an explicit class $\alpha_{\ramp}$. 
 If $g=0$, then $\alpha_{\ramp} =1$ and if $g=1$, then $\alpha_{\ramp} \in \divR[2](\PMSDS)$.
We describe  $W_i(\ramp)$ and $\alpha_{\ramp}$ explicitly in  \Cref{prop-g-0,prop-g-1}.
\end{proposition}
The coefficients $\ell_{\lG}$ will be defined in \Cref{sec-lgs}.

 \subsection*{Lifting admissable cover cycles to the log Chow ring}
 The log Chow ring  $\logCH(\Mgnbar)$ ring of $\Mgnbar$ can be constructed via logarithmic geometry or alternatively as the colimit of the Chow rings of iterated blowups of boundary strata of $\Mgnbar$, see for example \cite{log-chow} for an introduction. Sometimes tautological classes become expressible more easily in the log Chow ring. For example it was shown in \cite{log-chow} that  the top $\lambda$-class $\lambda_g$ does not lie in the subalgebra of $\CH(\Mgnbar[g])$ generated by divisorial classes for $g\geq 2$ but if lifted to the log Chow ring $\lambda_g$ lies in the divisorial subring $\operatorname{divlogCH}^*(\Mgnbar[g])$.

In \cite[Thm. 5.13 ]{tale-moduli} the moduli space of multi-scale differentials $\PMSDS$ was related to an iterated blowup coming from log geometry.
We can thus use  closures of strata of $\ramp$-exact differentials $\Exc$ to define a lift of admissable cover cycles to the log Chow ring.
We have shown that the class of the closure  of $\Exc$ lies in the divisorial subring of $\CH(\PMSDS)$ and it seems an  interesting question to determine whether the lift lies in the divisorial subring $\operatorname{divlogCH}^*(\Mgnbar[g])$.

\subsection{Motivation from Teichm\"uller dynamics and bialgebraic geometry}
Strata of exact differentials are examples of a more general class of subvarieties of strata $\PSt$.
The stratum $\St$ has a linear structure given by local period coordinates and {\em linear subvarieties} of $\St$ are algebraic subvarieties which are locally given by linear equations in period coordinates.
If the linear equations are defined over the real numbers, as it happens for $\Excmu$, then a linear subvariety is invariant for the natural $\GL_+(2,\RR)$-action on strata. In a holomorphic stratum $\GL_+(2,\RR)$-invariant varieties are unions of finite orbit closures and the intersection theory of an orbit closure is (conjecturally) closely related to the dynamics of the $\GL_+(2,\RR)$-action (see for example \cite[Conjecture 4.3]{chen-moeller-sauvaget}).
Additionally, in \cite{euler-linear} the Euler characteristic of a linear subvariety is expressed as intersection number on the moduli space of multi-scale differentials.
If the class of an orbit closure can be explicitly computed in the Chow ring of $\PMSDS$ these formulas can be evaluated effectively. This motivates the question whether classes of orbit closures lie in the tautological ring. One way to produce an orbit closure from a given one is using covering constructions.  But tautological classes are not preserved under branched coverings, for example the bielliptic locus is in general not tautological, see \cite{graber-pandharipande} and can be realized as a covering construction of a stratum in genus one.
\begin{question}
Are pushforwards of classes of affine invariant submanifolds H-tautological in $\Mgnbar$ (in the sense of \cite{Lian})?
\end{question}
In this paper we deal with two classes of linear subvarieties, strata of exact differentials and residue subspaces, both of which only occur in strata of meromorphic differentials. In both cases the classes happen to lie in the (classical) tautological ring. 

Linear subvarieties are examples of bialgebraic subvarieties of strata (see \cite{Klingler-Lerer} for an introduction to these circles of ideas). Since bialgebraic varieties are very special and rare, one expect their classes to have special numerical properties as well and it seems interesting to investigate when these classes are tautological.

\subsection{Residueless differentials}\label{sec-residues}
Our methods not only work for strata of exact differentials  but also for linear subvarieties defined by linear equations among residues. These can be defined as follows.
On a stratum of meromorphic differentials $\St$ there exists a residue map
\[
R: \St\to \CC^{r}, 
\]
where $r$ is the number of poles of $\St$, sending a differential to the residues at the marked poles.
A {\em residue subspace} is an algebraic variety of the form $R^{-1}(V)$, where $V\subseteq \CC^r$ is some linear subspace.
If $V$ is defined over the real numbers, $R^{-1}(V)$ is a $\GL_+(2,\RR)$-invariant subvariety.
In the case where $V$ is the zero subspace we call the resulting variety the stratum of residueless differentials.
Their classes in the cohomology ring of $\Mgnbar$ have been recently considered in \cite{residueless} to construct a partial cohomological field theory related to the KP hierarchy.

 \Cref{thm-main-general} implies that Chow classes of residue subspaces lie in the divisorial tautological ring but in this special case the proof is easier and we also obtain a closed formula.

\begin{proposition}\label{prop:residue-subspaces} Let $Z$ be the closure of a residue subspace in  $\PMSDS$. Then there exists a closed formula 
\[
Z = (-1)^{\codim Z} \prod_{i} (\xi + \sum_{\lG\in W_i(Z)} \ell_{\lG}\bd[\lG]) \in \CH(\PMSDS)
\] in the divisorial tautological ring. Here $W_i(Z)$ is a collection of two-level graphs that we describe explicitly in 
\Cref{prop:residue-closed}.
\end{proposition}

\subsection{Outline}
We briefly sketch the proof idea for \Cref{thm:main}, highlighting the technical difficulties that we encounter along the way.
On the stratum $\PSt$ we can find a vector bundle $\calE$ with a section $s$ whose zero locus is exactly the stratum of exact differentials $\Excmu$. In particular the class of $\Excmu$ is given by the top Chern class of $\calE$.
The first step in extending the computation to the compactification $\PMSDS$ is to extend $\calE$ and $s$ to $\PMSDS$. This is done in \Cref{sec:Deligne}.
The extended section $\cl{s}$ contains the closure of $\Excmu$ but it also contains other extraneous components.
The extraneous components can be described explicitly and are all related to exact differentials on nodal curves.
The main technical difficulty is that some extraneous components are not of expected codimension.
We deal with this difficulty by blowing up all extraneous components in a carefully chosen order. On the blowup the extraneous components become divisorial and the extended section $\cl{s}$ vanishes with a certain multiplicity. After dividing by the defining equation of the divisors we obtain a new section of a different vector bundle and it turns out that now the zero locus only consists of the proper transform of $\cl{\Excmu}$. We then need to blow down all the exceptional divisors and use  intersection theory of blowups to obtain the class of $\cl{\Excmu}$.
Another technical difficulty is that some of the extraneous components are not smooth but instead have non-reduced stack structures and therefore the resulting blowup might not be smooth as well. By a careful analysis of the local defining equations we show that the singularities are mild enough so that the Chow groups still have an intersection product and the classes of proper transforms can be computed.

\subsection*{Acknowledgments}
We would like to thank Samuel Grushevsky and Martin M\"{o}ller for helpful discussions. We also would like to  thank Matteo Costantini, Johannes Schmitt and Rahul Pandharipande for comments on an earlier draft.

\section{Preliminaries}\label{sec:prelim}

\subsection{Exact differentials}
Given a rational function $f:X\to \PP^1$ from a Riemann surface $X$ of genus $g$, the {\em associated exact differential} is  $df=f^*dz$ and we can recover $f$ by integrating.
By acting on $\PP^1$ with M\"obius transformations we can produce isomorphic rational functions with different exact differentials. In order to associate to the isomorphism class of $f$ a unique exact differential, at least up to rescaling, we need to normalize $f$ suitably.
By choosing a marked point $p_1$ on $X$, we can normalize $f$ such that$f(p_1)=\infty$. This determines $f$ up to rescaling.

Exact differentials  characterized in terms of their periods. A meromorphic differential $\omega$ on $X$ is exact if and only if all its absolute periods ~$\int_{\gamma}\omega$ for ~$\gamma\in H^1(X\setminus P(\omega);\ZZ)$ vanish, where $P(\omega)$ denotes the set of poles of $\omega$.
In that case the rational function is given by $f(x)=\int_{x_0}^x \omega$ for some choice of base point $x_0$ and some choice of path from $x_0$ to $x$.
Later we will see that that the vanishing of absolute periods can be described as the zero locus of a section of a vector bundle on the stratum of meromorphic differentials, which will the basis of our computation.

Now we address how to determine the ramification of a rational function from its exact differential.
The ramification multiplicities of $f$ are determined by the orders of zeros and poles of $df$. In particular we have
\[
\mult_{p} f = \begin{cases} \ord_p df+1 & \text{ if } f(p)\neq\infty,\\
-\ord_p df -1 & \text{otherwise}.\end{cases}
\]

A {\em ramification profile} $\ramp=(d;\lambda_1,\ldots,\lambda_k)$ consists of a natural number $d$ and a tuple of partitions $\lambda_i= (\lambda_{i,1},\ldots,\lambda_{i,l(\lambda_i)})$ satisfying the Riemann-Hurwitz condition
\[
\sum_{i,j} \left( \lambda_{i,j}-1\right) = 2d + 2g(X)-2.
\]
A marked curve $(X,Q=(q_{i,j}),f)$ with a rational function $f$ of degree $d$ is said to have ramification profile $\ramp$ if at all marked points we have $\mult_{q_{i,j}} f = \lambda_{i,j}$ and if two points $q_{i,j}, q_{i,k}$ correspond to the same partition, they lie in the same fiber of $f$, i.e., $f(q_{i,j})= f(q_{i,j})$. In terms of periods of $df$ this translates to 
\[
\int_{q_{i,j}}^{q_{i,k}} df=0.
\] 
The integration here is performed on any path on $X$ connecting $q_{i,j}$ and $q_{i,k}$ and is independent of the choice of such a path, since we already know that all absolute periods of $df$ are zero.

\subsection{Strata of exact differentials}
By  passing to exact differentials, rational functions with a given ramification profile $\ramp$ can be described in terms of meromorphic differentials with prescribed orders and zeros, together with  additional linear constraints on certain periods.

We let $\mu$ be a tuple of  order of zeros and poles of an exact differential with ramification profile $\ramp$. In general there is ambiguity because of the choice of a pole of $f$. However, we will always assume that the first partition $\lambda_1$ corresponds to the poles of $f$. With this convention $\ramp$ determines  a unique  partition $\mu$, which we call  the {\em partition associated to } $\ramp$.
For example if we consider  hyperelliptic curves the ramification profile is $\ramp=(2^{2g+2})$ and the associated partition is $\mu=(-3,1^{2g+1})$, where we used exponential notation to denote  repetitions.

Our goal now is to define moduli spaces for exact differentials.
The stratum $\St$ parametrizes tuples $(X,Q=(q_1,\ldots,q_n),\omega)$ of  marked Riemann surfaces $X$ of genus $g$ together with a meromorphic differential $\omega$ vanishing to  order $\mu_i$ at $p_i$. Most of the time we only consider differentials up to rescaling, which are parametrized by the projectivized stratum $\PSt=\St/\CC^*$.
The {\em stratum of  exact differentials}
\[
\begin{split}
\Excmu&
=\left\{(X,\omega)\,:\, \int_{\gamma}\omega=0 \text{ for all }\gamma \in H_1(X\setminus P(\omega);\ZZ)\right\}\subseteq \PSt
\end{split}
\]
is a codimension $2g+|P(\omega)|-1$ subvariety of the projectivized stratum.

Differentials in $\Excmu$ are exact and have prescribed ramification multiplicities. To further specify the ramification profile $\ramp$ we thus need to pass to a further subvariety, taking into account which  points lie in the same fiber.
We write $\ramp=(\lambda_1,\ldots,\lambda_m)$ with $\lambda_i=(\lambda_{i,1},\ldots,\lambda_{i,l(\lambda_i)})$.
Sometimes it will be useful to re-index the marked points $Q=(q_1,\ldots,q_n)$ by $(q_{1,1},\ldots, q_{1,l(\lambda_1)},\ldots,q_{m,1},\ldots,q_{m,l(\lambda_m)})$ such that the marked point $q_{i,j}$ corresponds to the entry $\lambda_{i,j}$ in the partition $\lambda_i$. We say that $(q_{i,1},\ldots, q_{i,l(i)})$ are the {\em  marked points lying
in the $i$-th fiber} of $\ramp$.

We can now define the {\em stratum of $\ramp$-exact differentials} 
\[
\Exc:=\left\{(X,Q,\omega)\in\Excmu\,:\,\int_{q_{i,j}}^{q_{j,k}} \omega =0 \text{ for all $1\leq i \leq m, 1\leq j \leq l(\lambda_i)$}\right\}.
\]

The expected codimension of $\Exc$ inside the stratum is \[d_{\ramp}=2g+P(\omega)-1+\sum_{i=1}^m \left(l(\lambda_i)-1\right).\]

\section{Local systems on strata}\label{sec:localsystems}

Over the stratum $\PSt$ the relative homology groups $H_1(X\setminus P(\omega),Z(\omega);\CC)$ assemble to a $\CC$-local system. Here $P(\omega)$ and $Z(\omega)$ denote the set of poles and zeros of $\omega$, respectively. We let $\calH$ be the associated flat vector bundle, which comes with  the {\em evaluation morphism}
\[
\ev:\calO(-1)\otimes \calH\mapsto \calO,\, \omega\otimes\alpha\mapsto \int_{\alpha} \omega.
\]

Our goal is to study subvarieties of strata defined by the vanishing of periods $\int_{\alpha} \omega$. In order to formalize this idea, let $\calV\subseteq \calH$ be a sub-local system. Here and in the following we do not distinguish between a local system and the associated flat vector bundle. We let \[\ev_{\calV}:\calO(-1)\otimes\calV\to\calO\] be the restriction of the evaluation section to cycles in $\calV$ and  then define the {\em annihilator of $\calV$} to be 
\[
\Ann{\PSt}:= Z(\ev_{\calV})=\left\{(X,\omega)\in\PSt: \int_{\alpha} \omega=0 \text{ for all } \alpha\in \calV_{(X,\omega)}\right\}\subseteq \PSt.
\]
The zero locus $\Ann{\PSt}$ is either empty or of codimension $\rank({\calV})$ in the stratum $\PSt$.  This can be checked for example in local period coordinates.
% In particular $\Ann{\PSt}$ always has the expected codimension.

In the sequel we will mainly focus on three examples of local systems $\calV$, which we will now discuss.

The first one is the local system of absolute homology $\calH_{abs}$ whose fiber over a point $(X,\omega)$ is $H_1(X\setminus P(\omega);\CC)$.
The zero locus $\Ann[({\calH_{\abs}})]{\PSt}$ coincides with the stratum of exact differentials $\Excmu$. The expected codimension of $\Excmu$ is $\rank(\calH_{\abs})=2g+|P(\omega)|-1$.

If we want to take the ramification profile of a rational function into consideration, we need to require certain relative periods to vanish and thus need to pass a local system containing $\calH_{\abs}$. This will be our second example.
Given a ramification profile $\ramp$ we define  \[
H_{\ramp}(X,\omega)\subseteq H_1(X\setminus P(\omega),Z(\omega);\CC)\]
 to be  the smallest subspace containing the absolute homology $H_1(X\setminus P(\omega);\CC)$ as well as all paths such that both  endpoints $q,q'$ lying in the same fiber, i.e., there exist $i,k,k'$ such that $q=q_{i,k},q'=q_{i,k'}$. In particular we have $\calH_{\abs}\subseteq \calH_{\ramp}\subseteq \calH$.
The annihilator of $\calH_{\ramp}$ consists of $\ramp$-exact differentials, in other words
\[
\Ann[(\calH_{\ramp})]{}=\Exc.
\]
The expected codimension in this case is $2g+|P(\omega)|-1+\sum_{i=1}^m \left(l(\lambda_i)-1\right)$, where we recall that $l(\lambda_i)$ is the number of parts in the partition $\lambda_i$.

The last example we will consider is related to linear relations among residues.
First consider the local system $\calH_{\res}$ with fiber $\ker( H_1(X\setminus P(\omega);\CC)\to H_1(X;\CC))$.
The bundle $\calH_{\res}$ is trivial and for any subspace $R\subseteq (\calH_{\res})_{(X,\omega)}$ of a fiber we consider the associated vector bundle $\calR\subseteq \calH_{\res}$. 
The annihilator of $\calR$ consists of all differentials whose residues satisfy the linear equations given by $\calR$.
A special case is the case $\calR=\calH_{\res}$ which yields the {\em stratum of residueless differentials} $\PMSDS^{\res}$.

The annihilator of any local system $\calV\subseteq \calH$ is the zero locus of a vector bundle and has the expected codimension. From this we can determine the class of $\Ann{}$ as the top Chern class of $\calV^*\otimes \calO(-1)$. The Chern classes of a flat vector bundle vanish in cohomology by Chern-Weil theory. For the examples $\calH_{\abs},\calH_{\ramp},\calR$, which we will focus on the sequel, the same is true also in the Chow ring.

\begin{definition}
A local system $\vb\subseteq\calH$ on $\PSt$ is {\em globally defined}
if either $\vb\subseteq \calH_{\res}$  or $\calH_{\abs}\subseteq \vb$.
\end{definition}

In particular $\calH_{\abs},\calH_{\ramp},\calR$ are all globally defined local systems.

\begin{remark}
If $\vb\subseteq \calH_{\res}$ then $\vb$ is a trivial local system.
On the other hand if $\calH_{\abs}\subseteq \vb$, then 
$\dfrac{\vb}{\calH_{\abs}}$ is trivial. 
\end{remark}

\begin{lemma} Let $\calV$ be any globally defined local system. Then $\calV$ has vanishing Chern classes in the Chow ring of $\PSt$.
\end{lemma}
\begin{proof}
 Firstly, if the bundle is trivial, it has zero Chern classes.
 Otherwise we can write $\vb$ as an extension
 \[
 0\to  \calH_{\abs} \to \vb \to  \calO^{n+1} \to 0 
 \]
 of $\calH_{\abs}$ by a trivial bundle.

Similarly, $\calH_{\abs}$ is an extension of $\calH_{\res}$ by the bundle $\calH_1$ with fiber $H_1(X;\CC)$ by a trivial bundle. Finally, the bundle $\calH_1$ sits in a short exact sequence

\[
0\to \calE \to\calH_1\to\calE^*\to 0,
\]
where $\calE$ is the Hodge bundle over $\Mgn$. It then follows from Mumfords computation \cite{mumford-enumerative} that in the Chow ring of $\Mgn$ the identity $\ch(\calH_1)=\ch(E)+\ch(E^*)=0$ holds.
 The same thus holds for the pullback to $\PSt$.
\end{proof}

\begin{proposition}\label{prop:classStratum}Suppose that $\calV\subseteq\calH $ is a globally defined local system.
Then the  class of $\Ann{\PSt}$ in the Chow ring of $\PSt$ is given by
\[
[\Ann{\PSt}]=c_1(\calO(1))^{r}\in\CH[r](\PSt),
\]
where $r:=\rank\calV$.
\end{proposition}

\begin{proof}
Since $\Ann{\PSt}$ is the zero locus of a section of $\mathcal{H}om(\calO(-1)\otimes\calV,\calO)\simeq \calV^*\otimes\calO(1)$ of the expected codimension, the class $[\Ann{\PSt}]$ agrees with the top Chern class $c_r(\calV^*\otimes\calO(1))$.
Thus
\[
[\Ann{\PSt}]=c_r(\calV^*\otimes\calO(1))=
\calO(1)^r,\]
where in the last equality we used that $\calV$ has zero Chern classes.
\end{proof}

\subsection{Moduli space of multi-scale differentials}
The moduli space of multi-scale differentials $\PMSDS$ was constructed in \cite{BCGGMsm} and compactifies the projectivized stratum $\PSt$. The space $\PMSDS$ is a smooth projective \DM stack and the boundary  is a normal crossing divisor.
Furthermore, the boundary $\PMSDS\setminus \PSt$ is stratified by level graphs, which we now recall.

\subsection{Level graphs}\label{sec-lgs}
 A level graph is a stable graph with a total order on the vertices (allowing equality) together with an integer for each half-leg, the order of vanishing, and an additional non-negative integer for each edge, called  the prong at $e$.
 Our notation slightly differs from \cite{BCGGMsm}, where the above is called an ``enhanced level graph". In this paper we exclusively deal with enhanced level graphs and thus refer to them simply as level graphs.
 We follow the convention from \cite{CMZ} and index the levels of $\lG$ by negative integers $\{0,-1,\ldots,-L\}$.
 The $i$-th level passage is a horizontal line above level $-i$ and below level $-i+1$ and thus level passages are indexed by $\{1,\ldots,L\}$.
 An edge between vertices of the same level is called {\em horizontal} and {\em vertical} otherwise. A level graph without horizontal edges is called {\em non-horizontal}.

 We denote $\LG$ the set of isomorphism classes of all level graphs and $\LG_k$ the subset of codimension $k$ level graphs. We also denote $\BIC\subseteq \LG_1$ the set of isomorphism classes of non-horizontal two level-graphs.
For $\lG\in\BIC$ we also define the factor $\ell_{\lG}$ to be the least common multiple of all prongs.

 \subsection{The boundary of $\PMSDS$}
The boundary of $\PMSDS$ is stratified by level graphs.
 The generic point of a boundary stratum $\bd[\lG]$ corresponding to a level graph $\lG$ consists of {\em twisted differentials} $(X,Q,\omega)$ which are compatible with $\lG$.
 The codimension of $\bd$ is the number of level passages, i.e., the number of level $-1$, plus the number of horizontal edges and thus equal to the codimension of the level graph.
 A feature of $\PMSDS$, first noticed in \cite{CMZ}, is that the non-horizontal boundary $\cup_{\lG\in\BIC} \,\bd[\lG]$ is a {\em simple} normal crossing divisor.

\subsection{Tautological rings}
\begin{convention}
All Chow rings in this paper are considered with rational coefficients.
\end{convention}
The tautological ring $\tr{\Mgnbar}$ is the smallest subring of the Chow ring which is closed under forgetful maps and gluing maps. A set of additive generators for $\tr{\Mgnbar}$ consists of decorated strata classes, which can be described as follows. Given a stable graph $\sG$ set $\Mgnbar[\sG]= \prod \Mgn[g_v,n_v]$ where $g_v$ and $n_v$ are the genus and the number of half-legs adjacent to $v$, respectively. For every polynomial $\alpha$ in $\psi$-classes in $\CH(\Mgnbar[\sG])$ the decorated strata class $[\sG,\alpha]$ is defined by 
\[
[\sG,\alpha]:= \zeta_{\sG,*}\alpha,
\]
where  $\zeta_{\sG}:\Mgnbar[\sG]\to \Mgnbar$ is the clutching morphism.

The tautological ring of the moduli space of multi-scale differentials $\PMSDS$ was first introduced in \cite{CMZ} and has a similar description. Given a level graph $\lG$ we have induced partitions $\mu_v$ for every vertex $v\in V(\lG)$.
Additionally there are induced residue conditions $\mathfrak{R}_i$ for every level coming from the matching residue condition at horizontal nodes and the global residue condition.
We can then form the disconnected stratum
\[
\St[(\mu_{[i]})]:=\prod_{\ell(v)=i} \St[(\mu_v)]
\]
and the {\em generalized stratum} is the subspace $\calH(\mu_{[i]})^{\mathfrak{R}_i}\subseteq \St[(\mu_{[i]})]$ satisfying all residue conditions in $\mathfrak{R}_i$.
The projectivized disconnected stratum $\PSt[(\mu_{[i]})]^{\mathfrak{R}_i}$ is the quotient by the diagonal $\CC^*$-action.
As remarked in \cite[Prop. 4.2]{CMZ}, the construction for multi-scale differentials can also be applied to generalized strata to obtain a compactification $B_{\lG}^{[i]}:=\PMSDSdisc{g_{[i]},n_{[i]}}{(\mu_{[i]})}$ of $\PSt[(\mu_{[i]})]^{\mathfrak{R}_i}$.
We then set 
\[
B_{\lG}= \prod_{i=-L}^{0} B^{[i]}_{\lG}.
\]
The product $B_{\lG}$ comes with a forgetful map $f_{\lG}: B_{\lG}\to \Mgnbar[\sG]$, where $\sG$ is the underlying stable graph of $\lG$.  In \cite{CMZ} the authors construct a commutative diagram
\begin{equation}\label{diag:roof}
\begin{tikzcd}
& D^s_{\lG}\ar[dr,"c_{\lG}"] \ar[dl,"p_{\lG}",swap] &&\\
B_{\lG}\ar[d,swap,"f_{\lG}"] && D_{\lG}\ar[r]&\PMSDS\ar[d,"\rho"]\\
\Mgnbar[\sG]\ar[rrr,"\zeta_{\sG}",swap]&&& \Mgnbar
\end{tikzcd}
\end{equation}
where $c_{\lG}$ and $p_{\lG}$ are finite maps and $D^s_{\lG}$ is Cohen-Macaulay.
The tautological ring $\tr{\PMSDS}$ is the smallest subring of $\CH(\PMSDS)$ containing the fundamental classes of all generalized strata and stable under clutching maps $c_{\lG,*}p_{\lG}^*$ for all non-horizontal level graphs. It was shown in \cite[Thm. 1.5]{CMZ} that an additive generating set of $\Rtr(B)$ is given by 
\[
c_{\lG,*}p_{\lG}^*f_{\lG}^*\alpha,
\]
where $\alpha\in \CH(\Mgnbar[\lG])$ is a polynomial in $\psi$-classes and $\lG$ runs over all non-horizontal level graphs. 

\begin{remark} In \cite[Sec. 8]{CMZ} various tautological rings are constructed. In this paper we only consider the non-horizontal version, the smallest of the subrings in (loc.cit.).
\end{remark}

\begin{proposition}\label{prop:taut}The forgetful map $\rho:\PMSDS\to\Mgnbar$ preserves tautological rings, i.e.,
\[
\rho_*\tr{\PMSDS}\subseteq \tr{\Mgnbar}.
\]
\end{proposition}

\begin{proof}
Since $\tr{\PMSDS}$ is generated by classes of the form $c_{\lG,*}p_{\lG}^*f_{\lG}^*\alpha$, it suffices to show that\[
\rho_*c_{\lG,*}p_{\lG}^*f_{\lG}^*\alpha
\]
is tautological. This follows from the commutativity of \cref{diag:roof} and the fact that strata classes are tautological \cite{universal-dr}.
\end{proof}

We stress that the inclusion is strict in general for dimension reasons.

\subsection{Profiles and boundary strata}
The boundary of a generalized stratum $B$ is naturally stratified by level graphs, but we will organize the boundary strata somewhat differently to have a better behaved intersection theory.
The set of non-horizontal two level graphs admits a partial ordering, see \cite[Section 5]{CMZ}, which we now recall.
Given any level graph $\lGh$ we define $\delta_i(\lGh)$ to be the $2$-level graph, which is obtained by smoothing all edges, except the ones crossing the $i$-th level passage.
Let $\lG,\lG'\in\BIC$ be two different non-horizontal $2$-level graphs. We say that $\lG\succ \lG'$ if there exists a non-horizontal $3$-level graph $\lGh$ with $\delta_1(\lGh)=\lG, \delta_2(\lGh)=\lG'$.
This is well defined by \cite[Prop. 5.1]{CMZ}.

A profile $\pf=(\lG_1,\ldots,\lG_k)$ is an ordered tuple of non-horizontal two-level graphs ordered such that $\lG_i\succ\lG_{i+1}$. We call $L(\pf):=k+1$ the number of levels of $\pf$ and define the factor $\ell_{\pf}:=\prod_{i=1}^k \ell_{\lG_k}$, where we recall that $\ell_{\lG_k}$ is the product of all prongs of $\lG_k$.
 The collection of all profiles  of $B$ is denoted by $\PF$.
 Every $(k+1)$-level-graph $\lGh$ has an associated profile $\delta(\lGh)=(\delta_1(\lGh),\ldots,\delta_{k}(\lGh))$.
 
 For a non-horizontal two level graph $\lG$ we usually write $\lG$ instead of $(\lG)$. The empty profile $\star$ corresponds to the unique level graph  with one vertex.

The boundary stratum associated to a profile is 
\[
\bd[\pf]=\bigcup_{\lGh\,:\,\delta(\lGh)=\pf} \bd[\lGh].
\]
The codimension of $\bd[\pf]$ is $L(\pf)-1$.
It was shown in \cite{CMZ} that for a profile $\pf=(\lG_1,\ldots,\lG_k)$ we have $\bd[\pf]= \cap_{i=1}^k \bd[\lG_i]$. In particular $\bd[\pf]$ is a complete intersection of codimension $i$ and $[\bd[\pf]]=\prod_{i=1}^k [\bd[\lG_i]]\in\tr{B}$.

Given two profiles $\pf_1,\pf_2$ we say that $\pf_1$ and $\pf_2$ are comparable if the intersection $\bd[\pf_1]\cap\bd[\pf_2]$ is non-empty and for two comparable profiles we define the sum $\pf_1+\pf_2$ to be the profile of some (or equivalently any) level graph in $\bd[\pf_1]\cap\bd[\pf_2]$. We say $\pf_1=(\lG_1,\ldots,\lG_k)$ and $\pf_2=(\Delta_1,\ldots,\Delta_l)$ are complementary if the sets $\{\lG_1,\ldots,\lG_k\}$ and $\{\Delta_1,\ldots,\Delta_l\}$ are disjoint and $\pf_1,\pf_2$ are comparable.
In this case we have \[
\ell_{\pf_1+\pf_2}=\ell_{\pf_1}\ell_{\pf_2}.
\]
If $\pf_1,\pf_2$ are complementary we write $\pf_1\oplus\pf_2=\pf_1+\pf_2$ and say that  $\pf_1\oplus\pf_2$ is a direct sum. For example if $\pf=(\lG_1,\ldots,\lG_k)$ then $\pf=\bigoplus_{i=1}^k \lG_i$.
If $\pf[Q]=\pf_1\oplus \pf_2$ is a direct sum, then the intersection 
\[
\bd[{\pf[Q]}]=\bd[\pf_1]\cap\bd[\pf_2]
\]is a transversal intersection and in particular $[{\bd[{\pf[Q]}]}]=[\bd[\pf_1]]\cdot [\bd[\pf_2]]\in\trd{B}$.

We say $\pf$ degenerates to $\pf'$ or $\pf'$ is a {\em degeneration} of $\pf$ if $\bd[\pf']\subseteq\bd[\pf]$. This is equivalent to being able to write $\pf'=\pf\oplus \pf[R]$ for some profile $\pf[R]$. We also say that $\pf$ is an {\em undegeneration} of $\pf'$ and we write $\pf\degen\pf'$. For a level graph $\lGh$ we say that $\lGh$ is a degeneration of a profile $\pf'$ if the profile $\pf(\lGh)$ of $\lGh$ is a degeneration of $\pf'$.

Similarly, we say that $\pf'$ is a degeneration of $\pf$ by splitting level $i$, if we have
$\pf= (\lG_1,\ldots,\lG_n)$ and $\pf' = (\lG_1,\ldots,\lG_{-i}, \Delta_1,\ldots,\Delta_k, \lG_{-i+1},\ldots,\lG_n)$. We let $\PF[(\pf,i)]$ be the collection of all profiles, which are degenerations of $\pf$ by splitting level $i$.

\subsection{The divisorial tautological ring} Let $B$ be a generalized boundary stratum.
Instead of the whole tautological ring we will mostly work in the {\em divisorial tautological ring} $\trd{B}$, the subring  of the tautological ring $\tr{B}$ generated by codimension $1$ classes in $\tr{B}$.
For us the most important divisorial classes are $\psi$-classes, boundary divisors $\bd[\lG]$ where $\lG\in\BIC$ and the first Chern class $\taut=c_1(\calO_{\PSt}(-1))$ of the tautological line bundle $\calO_{\PSt}(-1)$ whose fibers is spanned by the differential $\omega$.

\begin{proposition}
If the generalized boundary stratum $B$  contains no simple poles, then a multiplicative generating set for $\trd{B}$ is given by \[
\taut\text{ and } D_{\lG} \text{ for } \lG\in \BIC.
\]
On the other hand, if $B$ contains simple poles, then a multiplicative generating set for $\trd{B}$ is given by $ D_{\lG} \text{ for } \lG\in \BIC$ and $\psi$-classes $\psi_j$ where $j$ runs over all marked simple poles.
\end{proposition}

\begin{proof}
By \cite[Thm 1.5]{CMZ} $\trd{B}$ is generated by $\bd[\lG]$ for $\lG\in\BIC$ and $\psi_j$ where $j$ runs over all marked points.
It then follows from \cite[Prop 8.2]{CMZ} that $\psi_j$ can be expressed in term of boundary divisors and $\taut$, as long as the $j$-th marked point is not a simple pole.
The same relation also show that $\taut$ is a linear combination of boundary divisors if $B$ contains simple poles.
\end{proof}

In the sequel we will mostly be interested in the case of exact differentials, in which case simple poles cannot appear.

For a boundary stratum $\bd$ we let $\trd{\bd}$ be the pullback of $\trd{B}$ along the inclusion $ \bd\hookrightarrow B$.
On $\bd$ there is a collection of line bundles $\taut_{\pf}^{[i]}$ which are generated by the differentials $\omega^{[i]}$ on each level, see also \cite[Sec. 4.3]{CMZ}. The line bundle $\taut_{\pf}^{[0]}$ is the restriction of $\taut$ and thus contained in $\trd{\bd}$. We will need later that in fact all line bundles $\taut_{\pf}^{[i]}$ are pullbacks of tautological divisorial classes on $B$.

\begin{proposition} Let $\pf$ be a profile.
Then \[
\taut_{\pf}^{[i]}\in \trd[1]{\bd}
\]for $i= -L(\pf),\ldots,0$.
\end{proposition}

\begin{proof}

Fix $i<0$ and let $\lG:= \delta_{-i}(\pf)$. Then by \cite[Lemma 7.6]{CMZ} we have
\[
\taut^{[i]}_{\pf} = j_{\pf,\lG}^*(\taut^{[-1]}_{\lG}),
\]
where $j_{\pf,\lG}: \bd\hookrightarrow \bd[\lG]$ denotes the inclusion.
It thus suffices to show the claim for a two-level graph $\lG\in\BIC$.
In this case it follows from \cite[Lemma 7.1]{CMZ} that

\[
\ell_{\lG}j_{\lG}^*[\bd[\lG]]=\ell_{\lG}c_1(\nb{\bd[\lG]}{B})= -\taut^{[0]}+\taut^{[-1]} - \sum_{\lG'} \ell_{\lG'} j_{\lG'}^*[\bd[\lG']]
\]
where the sum runs over all two level graphs with  $\lG'\succ\lG$.
\end{proof}

\subsection{Generalized boundary strata}\label{sec-generalized-boundary}
The boundary strata $\bd$ are smooth substacks but we will be forced to consider
infinitesimal thickenings of the boundary strata $\bd$. Let $\pf=(\lG_1,\ldots,\lG_m)$ and $ \pfQ=(\lG_{m+1},\ldots,\lG_n)$ be complementary profiles.
Local equations for the boundary stratum $\bd[\pf+\pfQ]$ are given by 
\[
t_{\lG_1}=\ldots=t_{\lG_n}=0,
\]
where  $t_{\lG_i}$ is a transversal coordinate for the  divisor $\bd[\lG_i]$.

\begin{definition}
We define the {\em generalized boundary stratum} $\bd[\pf,\pfQ]$ to be the thickening of $\bd[\pf+\pfQ]$, which is locally defined by the ideal
\[
(t_{\lG_1},\ldots,t_{\lG_m}, t_{\lG_{m+1}}^{\ell_{\lG_{m+1}}},\ldots,t_{\lG_{n}}^{\ell_{\lG_{n}}}).
\]

\end{definition}

In particular $\bd[\pf,\pfQ]$ is a complete intersection and 
\[
[\bd[\pf,\pfQ]] =\ell(\pfQ) \prod_{i=1}^n [\bd[\lG_i]]\in \trd{B}.
\]

\subsection{Comparing the cohomology at the boundary}\label{sec-coho-boundary}
Our goal is extend the computation of $[\Ann{\PSt}]$ to the Chow ring of $\PMSDS$. We will do this by extending both the vector bundle $\vb$ and the evaluation section. As preparation we need a way to compare the relative homology at a twisted differential in some boundary stratum $\bd\subseteq\PMSDS$ to the homology at a nearby smooth Riemann surface in $\PSt$.
This has been carried out in general in \cite{Fred}, here we only recall the basic notions we need, with the notation adapted to our situation.

Let $x=(X,\omega)$ be a twisted differential compatible with a level graph $\lGh$, i.e., $x$ lies in the interior $\bd[\lGh]^{\circ}$ and denote by $\tilde{X}$ the normalization of the stable curve $X$ together with a marking of all preimages of nodes. We can then filter all connected components of $\tilde{X}$ as well as the  marked zeros and half-legs by level.
We let $\tilde{X}_{[i]}$ be the union of connected components corresponding to level  $i$ and $\tilde{Z}_{[i]}, \tilde{P}_{[i]}$
the marked points of level $i$, sorted into zeros and poles. Thus for example $\tilde{Z}_{[0]}$ consists of all marked zeros on the top level together with the preimages of all nodes  on the top level which do not correspond to simple poles.

We define $H_1^{[i]}(X,\omega):= H_1(\tilde{X}_i\setminus\tilde{P}_{[i]}, \tilde{Z}_{[i]};\CC)$. 
From now on all cohomology groups are considered with complex coefficients.
Note that the dimension $\dim H_1^{[i]}(X,\omega)$ only depends on the level graph $\lGh$.

Let $x'=(\Sigma,\omega')\in\PSt$ be a nearby flat surface with a collection of simple closed curves $\{\sigma_e, e\in E(\lGh)\}$, called {\em vanishing cycles}, such that $X$ is obtained by contracting the vanishing cycles. Away from the nodes we can realize $X$ as a subsurface of $\Sigma$ and thus consider $P=P(\omega), Z=Z(\omega)$ as a subset of $\Sigma$.
Vanishing cycles can be sorted into {\em horizontal} and {\em vertical} vanishing cycles depending on whether the corresponding node is horizontal or vertical. 

In order to relate periods along a degeneration with periods of the limiting twisted differential, we need to compare suitable relative homology groups of $\Sigma$ and $X$.
There are natural restriction maps from the relative homology on $\Sigma$ to the relative homology  of $X$, with special care taken at homology classes that intersect horizontal vanishing cycles.
We will only work with complex coefficients but all constructions in this section can also be carried out with integral coefficients.

We define the {\em non-horizontal cycles at level $0$ }to be 
\[
W_{[0]}:=\{ \gamma\in H_1(\Sigma\setminus P,Z;\CC)\,|\, \langle \gamma,[\sigma_e]\rangle =0 \text{ for any horizontal node $e$ of level $0$}\}.
\]

Here $\langle\cdot,\cdot\rangle$ denotes the  intersection pairing 
\[
 H_1(\Sigma\setminus P,Z;\CC) \times  H_1(\Sigma\setminus Z,P;\CC)\to \CC
\]
and $[\sigma_e]\in H_1(\Sigma\setminus Z,P;\CC)$ denotes the corresponding cohomology class.

On $W_{[0]}$ every cohomology class can be represented by a sum of simple closed curves disjoint from all horizontal vanishing cycles in top level, which we can then restrict a cohomology class in $H_1^{[0]}(X,\omega)$. This yields the specialization morphism
\[
\spc_{[0]}: W_{[0]}\to H_1^{[0]}(X,\omega).
\]

The kernel $\ker \spc_{[0]}$ consists of homology class that can be represented by classes in $X_{[-1]}$.
We set 
\[
W_{[-1]}:= \{\gamma\in H_1(\Sigma\setminus P,Z;\CC)\,|\, \langle \gamma,[\sigma_e]\rangle =0 \text{ for any horizontal node $e$ of level $-1$} \}.
\] 
On $W_{[-1]}$ we can now restrict a homology class to $X_{[-1]}$ and thus obtain the specialization morphism in level $-1$
\[
\spc_{[-1]}: W_{[-1]}\to H_1^{[-1]}(X,\omega).
\]

We thus have a three step filtration
\[
\{0\} \subseteq W_{[-1]}\subseteq \ker\spc_{[0]}\subseteq W_{[0]}\subseteq H_1(\Sigma\setminus P,Z)
\]
with graded pieces
\[
\begin{split}
&W_{[0]}/ \ker\spc_{[0]}\simeq H_1^{[0]}(X,\omega),\\
& W_{[-1]}\simeq H_1^{[-1]}(X,\omega).
\end{split}
\]

\begin{definition} \label{def:rlg} Suppose $x=(X,\omega)$ is a twisted differential compatible with a level graph $\lGh$.
We define 
\[
\begin{split}
&\vb_{\lGh,x}^{[0]}:= \spc_{[0]}(\calV_{x'}\cap W_{[0]}),\\
&\vb_{\lGh,x}^{[-1]}:= \spc_{[-1]}(\calV_{x'}\cap W_{[-1]}),
\end{split}
\]
where $\calV_{x'}$ is the fiber of $\calV$ at $x'$.
\end{definition}

The spaces $\vb_{\lGh,x}^{[0]}$ and $\vb_{\lGh,x}^{[-1]}$ can be interpreted as limits of homology classes as $\Sigma$ degenerates to $X$.

So far we have only treated the case of a two level graph, but everything can be generalized to an arbitrary profile $\pf$.
In that case one has two filtrations
\[
0\subseteq W_{[-L]}\subseteq \ker \spc_{[-L+1]}\subseteq\ldots\subseteq W_{[-1]}\subseteq \ker\spc_{[0]}\subseteq W_{[0]}\subseteq H_1(\Sigma\setminus P,Z),
\]
where $\spc_{[i]}$ now restricts a cohomology class in $W_{[i]}$ to the $i$-th level subsurface.
We define
\[
\vb_{\pf,x}^{[i]}:= \spc_{[i]}(W_{[i]}\calV_{x'}).
\]

On the interior $\bdint$ of a boundary stratum $\bd$, the vector spaces $\vb_{\pf,x}^{[i]}$ glue together to a flat vector bundle $\vb_{\pf}^{[i]}$.

\begin{remark}
The boundary stratum $\bd[\pf]$ can have several components, so in theory $\vb_{\pf}^{[i]}$ could have different ranks on each component. To see that the rank is indeed independent of the component, observe first that for any level graph $\lGh$ 
the rank of $\vb_{\lGh}^{[0]}$ can be computed from $\lG_1=\delta_1(\lGh)$, since we only need to know which paths are in level $0$ and which paths are in lower level.
In the next step we consider $\lG_2:=\delta_2(\lGh)$. The paths on top level of $\lG_2$ consists of all paths in  $\vb_{\lGh}^{[0]}\oplus \vb_{\lGh}^{[-1]}$, thus the dimension of $\vb_{\lGh}^{[-1]}$ only depends on $\lG_1$ and $\lG_2$.
By induction it follows that the dimension of $\vb_{\lGh}^{[i]}$ only depends on the profile $\delta(\lGh)=(\lG_1,\ldots,\lG_{L(\lGh-1)})$.
\end{remark}

\section{The class of strata of exact differentials}
For the rest of this section we fix a globally defined local system $\calV\subseteq \calH$ on a projectivized stratum $\PSt$.
We are now setting up our computation of the closure $\clAnn{}$ in the Chow ring of $\PMSDS$. 
The goal is the following, which has \Cref{thm:main} as a corollary.

\begin{theorem}\label{thm-main-general}
Let $\vb$ be a globally defined local system on a generalized boundary stratum $B$. Then the class of the closure $[\cl{\Ann{B}}]$ is contained in the divisorial tautological ring $\trd{\PMSDS}$.
Furthermore, the class can be computed explicitly.
\end{theorem}

The outline for this section is as follows. We will only work with a stratum $\PMSDS$ instead of a generalized boundary stratum, but the argument goes through verbatim in the more general case. We start by extending both the vector bundle $\calV$ as well as the evaluation section to $\PMSDS$ by passing to the Deligne extension $\overline{\calV}$ as carried out in \cite[Sec. 6.3]{CMZ}. The zero locus of the the extended section contains $\cl{\Ann{\PMSDS}}$ as well as several extraneous component of varying dimensions.  We determine the extraneous components explicitly as  zero loci of sections of  vector bundles on a generalized boundary stratum; similar to the original bundle $\vb$. This  sets up the beginning of a recursive computation.
%To control the additional components we pass to the complete flag variety $\Fl$, where all additional components become divisorial and thus can be taken out with a suitable multiplicity. The formula for the class of the stratum of exact differentials is then obtained by pushforward along the forgetful map  $\Fl\to\PMSDS$.

\subsection{The Deligne extension of $\calV$}\label{sec:Deligne}
The local system $\calV$ has unipotent monodromy along the normal crossing divisor $\PMSDS\setminus \PSt$.
In this situation the Deligne extension $\overline{\calV}$ is an extension of $\calV$ to $\PMSDS$ to a vector bundle on $\PMSDS$ admitting a connection with regular singularities along the boundary.
We will need a local description of $\overline{\calV}$. By choosing a local analytic coordinate system for $\PMSDS$ near the boundary, the complement $\PMSDS\setminus \PSt$ is isomorphic to $(\Delta^*)^k\times \Delta^l$.
Let $T_i$ be the monodromy of $\calV$ corresponding to the $i$-th generator of $(\Delta^*)^k$. We have $(I-T_i)^2=0$ and thus $N_i:= - \log T_i= I-T_i$.
 We choose a base point $x'\in (\Delta^*)^k\times \Delta^l$ and set $V=\calV_{x'}$.
 Any cycle $v\in V$ can be extended to flat section in a small neighborhood of $x$. In order to extend $v$ to a section on all of $(\Delta^*)^k\times \Delta^l$ we need to make $v$ monodromy invariant.
Consider the universal cover 
\[
\pi:\mathbb{H}^k\times \Delta^l \to (\Delta^*)^k\times \Delta^l, (z_1,\ldots,z_k,q_1,\ldots,q_l)\mapsto (e^{2\pi iz_1},\ldots, e^{2\pi iz_k},q_1,\ldots,q_l).
\]
The section $e^{2\pi i \sum_{i=1}^k z_iN_i}v$ of $\pi^*\calV$ is monodromy invariant and thus descends to a section $\overline{v}$ for $\calV_{|(\Delta^*)^k\times \Delta^l}$.
The Deligne extension $\overline{\calV}$ is the unique extension of $\calV$ that is locally over $\Delta^k\times \Delta^l$ trivialized by the sections $\overline{v}$ for $v\in \calV_{x'}$.
We call $\cl{v}$ the Deligne extension of the cycle $v$ and stress that it depends on the choice of local coordinates for $\PMSDS$ near the boundary.

The following was proven in \cite{CMZ} for the local system $\calH$ and then follows for any sub-local system $\calV\subseteq \calH$ by restriction.
\begin{proposition}
The evaluation morphism $\ev_{\calV}:\calV\otimes \calO_{\PSt}(-1)\to\calO_{\PSt}$ extends to a morphism
\[
\overline{\ev}_{\calV}:\overline{\calV}\otimes\calO_{\PMSDS}(-1)\to\calO_{\PMSDS}.
\]
Suppose $x=(X,\omega)\in\bd[\pf]^{\circ}$, $x'\in\PSt$ a nearby point and $v\in \calV_{x'}\cap W_{[0]}$. Then
\[
\overline{\ev}_{\calV}(\overline{v}\otimes\omega)= \int_{\spc_{[0]}(v)}\omega.
\]
In particular, $\overline{\ev}_{\calV}$ vanishes if $v\in \calV_{x'}\cap W_{[-1]}$.
\end{proposition}

It follows that the evaluation morphism vanishes identically at $x$ if and only if the periods over all cycles in $\vb_{\pf,x}^{[0]}$ vanishes.

Everything so far can be repeated for the vector bundles $\vb_{\pf}^{[i]}$ on $\bdint$. The boundary $\bd\setminus\bdint$ is a normal crossing divisor and both $\vb_{\pf}^{[i]}$ as well as the natural evaluation morphism
\[
\ev_{\pf,[i]}:\vb_{\pf}^{[i]} \otimes \calO_{\bdint}(-1)\to \calO_{\bdint}: \gamma\otimes\omega\mapsto \int_{\gamma} \omega,
\]
can be extended to $\bd$ using the Deligne extension $\cl{\vb}_{\pf}^{[i]}$.

\begin{proposition}\label{prop:vanishing} Let $\vb$ be a globally defined local system on a stratum $\PSt$.
The Deligne extension $\cl{\vb}_{\pf}^{[i]}$ has vanishing Chern classes in $\CH(\bd)$.

\end{proposition}

\begin{proof}
We first note that the statement is immediate in rational {\em cohomology} since the Deligne extension has a connection with nilpotent curvature.
To see the statement in the Chow ring we first consider the case of $B=\PSt$ an ordinary stratum and $\pf$ the empty profile.
In this case $\vb_{\star}^{[0]}$ is an extension of a trivial bundle by the bundle of absolute cohomology, which has trivial Chern classes by Mumfords computation \cite{mumford-enumerative} of the Chern classes of the Hodge bundle.

Now for the general case $\bd[\pf]$, it suffices to show that the statement for every connected component $\bd[\lGh]$ of $\bd[\pf]$. We consider the finite covering  $\widetilde{\bd[\lGh]}$ by marking all edges. On  $\widetilde{\bd[\lGh]}$ the bundle $\vb_{\pf}^{[i]}$ is the pullback of the Deligne extension from $\Mgn[\lGh]$, which has zero Chern classes by Mumfords computation again.
\end{proof}

Thus far we have only worked with a projectivized stratum but
all the local systems $\calH_{\abs},\calR,\calH_{\ramp}$ can be defined in an analogous way on the smooth part of  generalized stratum $B$. For any flat vector bundle $\vb$ on $B$ we can define $\vb^{[i]}_{\pf}$ on the smooth part of a boundary stratum $\bd$ and can be extended to $\bd$ using the Deligne extensions. The computation of \Cref{prop:vanishing} for globally defined local systems extends to this more general situation and shows that $c_i(\cl{\vb^{[j]}_{\pf}})=0$ for all $j$.

 \subsection{Coordinates near the boundary and log periods}
Using the Deligne extension of $\calH$ we can construct a local analytic (orbifold) coordinate system  $\PMSDS$ near  a boundary point $(X,\omega)\in \bd[\lGh]^{\circ}$ for some level graph $\lGh$.

First assume that $\lGh$ has no horizontal nodes.

We write $\omega = (\omega_i)_{i\in 0,\ldots, -L(\lG)}$.
For every level $i$ we choose a basis $a_i=(a_{i,1},\ldots,a_{i,l_i})$  for the homology $\calH_{\lGh}^{[i]}$ of level $i$.
For each cycle $a_{i,k}$ on $X$  we let $\cl{a}_{i,k}$ be a cycle on the Deligne extension $\cl{\calH}$ that specializes to $a_{i,k}$.
We also need the rescaling factor of level $i$
\[
\rescl:= \prod_{i=1}^{L(\lGh)} t_i^{\ell_i},
\]
where $\ell_i=\ell_{\delta_i(\lGh)}$ and $t_i$ is a coordinate transversal to the divisor $\bd[\delta_i(\lGh)]$.

We then define the {\em log periods of level $i$} $\alpha_i = (\alpha_{i,1},\ldots,\alpha_{i,l_i})$ by the formula 
\[
\alpha_{i,k} = \dfrac{1}{\rescl} \int_{\cl{a}_{i,k}} \Omega,
\]
where $\Omega$ is the universal family of differentials.

It was shown in \cite{Fred}
that 
\[
(t_1,\ldots,t_{L(\lG)}, \alpha_0,\ldots,\alpha_{L(\lG)})
\]
is a local coordinate system near $(X,\omega)$ and furthermore
\[
\alpha_{i,k}(X,\omega) = \int_{a_{i,k}} \omega.
\]

Let $\pf$ be some profile and $\lGh$  a level graph, which is a  degeneration of $\pf$.  We now describe the evaluation morphism 
\[
\cl{\ev}_{\pf}^{[i]}:\cl{\vb}_{\pf}^{[i]} \otimes \calO_{\bd}(-1)\to\calO_{\bd}
\]
 near a point in $\bd[\lGh]$ in terms of log periods.
 
 On $\lGh$ the $i$-th level has split into levels $j,\ldots, j-k$ for some $j,k$ and we choose a basis of cycles $a_l$ for $\vb_{\lGh}^{[l]},\, j \geq l \geq j-k$.
In order to compute the evaluation morphism $\cl{\ev}_{\pf}^{[i]}$ we need to work with a  frame for $\cl{\vb}_{\pf}^{[i]}$.
On the other hand to compute log periods we use a frame for $\cl{\vb}$. 
For a cycle $a_{l,m}$ of level $l$, the different extensions to cycles on the Deligne extensions $\cl{\vb}$ and $\cl{\vb}_{\pf}$ only differ in levels lower than $l$ and thus $(X,\omega)\in \bd[\lGh]^{\circ}$ has the same periods over both possible extensions.

 In a local frame the evaluation morphism $\cl{\ev}_{\pf}^{[i]}$ is given by 
\begin{equation}\label{eq-rescaled-evaluation}
\left(\top_{l,j}\alpha_{l,m}\right)_{j-k \leq l \leq j,m},
\end{equation}
where we recall
$\top_{l,j}:= \dfrac{\rescl[l]}{\rescl[j]} = \prod_{s=j-k}^{j} t_s^{\ell_s}$.

{\em Extension to horizontal nodes:}
The treatment of horizontal nodes is slightly different.
 For  every horizontal node $e$ of level $i$ we pick a cycle $\rho_e$ crossing $e$ once and not crossing any other horizontal nodes. 
 We can make it almost monodromy invariant by considering 
 \[
 \hat{e} = e - \sum_{e'\neq e} \langle \rho_e,\lambda_{e'}\rangle [\lambda_{e'}]\ln(s_{e'})
 \]
 The cycle $\hat{e}$ has only monodromy around $\lambda_e$.
 Then
 \[
 q_e := \exp\left({2\pi i\dfrac{\int_{\hat{e}} \Omega}{\int_{\lambda_e} \Omega}}\right)
 \]
 
is a coordinate.
(Our normalization differs slightly from \cite{CMZ}).

\subsection{The zero locus of the evaluation section}

Our next goal is to describe the zero locus of the evaluation morphism $\cl{\ev}_{\pf,[i]}:\vb_{\pf}^{[i]}\otimes\calO(-1)\to\calO$.
To lighten the notation we fix a projectivized stratum $B=\PSt$ and a globally defined flat vector bundle $\vb$.

We first need to introduce a few more spaces and infinitesimal thickenings, which will appear in the computation. We define 
\[
\bdexl{i} := \cl{\{ (X,\omega)\in\bd[\pfP,\pfQ]^{\circ}\,:\, \omega \in \ANN(\vb_{\pf+\pfQ}^{[i]}) \}},
\]
where $\bd[\pfP,\pfQ]^{\circ}$ is the open set corresponding to differentials with profile $\pfP+\pfQ$.
Here we implicitly pushed $\vb_{\pf+\pfQ}^{[i]}$ forward under the inclusion $\bd[\pf+\pfQ]\subseteq \bd[\pf,\pfQ]$,
where we recall the generalized boundary stratum $\bd[\pf,\pfQ]$ from \Cref{sec-generalized-boundary}.

If $\pfQ$ is the empty profile, we simply write $\bdexl[\pfP]{i}$.
Rephrased in this notation our main goal is to compute $\bdexl[\star]{0}$ but it turns out that in our recursive computation we will need to compute the classes of all subspaces $\bdexl{i}$.

The space $\bdexl{i}$ is an example of a linear subvariety in the sense of \cite{BDG}. Thus the description of \cite[Sec. 4]{BDG} allows to describe the equations near a boundary point explicitly.
In our situation the equations are particularly simply, since if a cycle in $\vb_{\pf}^{[i]}$ crosses a horizontal node, the corresponding vanishing cycle is also contained in $\vb_{\pf}^{[i]}$. 
Thus if on the open part $\bd[\pf,\pfQ]^{\circ}$ the defining equations for $\bdexl{i}$ are given by the vanishing of a collection of periods $a_1,\ldots, a_n$, the local equations near the boundary are given by the corresponding log period $\alpha_1,\ldots,\alpha_n$.
We now give an explicit description of the equations.

Let $\pf,\pfQ$ two complementary profiles and consider $\bdexl{i}\subseteq \bd[\pf,\pfQ]$. Fix $(X,\omega)\in \bd[\lGh]$ for some level graph  $\lGh$, which is a degeneration  of $\pf+\pfQ$ and let  $I\subseteq \{0,\ldots,-L(\lGh)\}$ be the collection of all levels of $\lGh$ that undegenerate to level $i$, when undegenerating $\lGh$ to the profile $\pf+\pfQ$.
%First assume that $(X,\omega)\in \bdint[\pf,\pfQ]$ is in the interior of $\bd[\pf,\pfQ]$.
We choose $\alpha= (\alpha_{1},\ldots,\alpha_{n})$ to be the collection of all log  periods over a basis of  $\oplus_{j\in I}\vb^{[j]}_{\lGh}$.
Then local equations for $\bdexl{i}\subseteq \bd[\pf,\pfQ]$ near $(X,\omega)$ are given by 
\begin{equation}\label{eq:bdexl}
\alpha_1=\ldots=\alpha_n =0.
\end{equation}

Since the log periods $\alpha$ are part of a local coordinate system, $\bdexl{i}$ is regularly embedded in $\bd[\pf,\pfQ]$.

Given a subset $J\subseteq \{0,\ldots, -L(\pf)\}$. We also define
\[
\bdexl{J} = \cap_{i\in J} \bdexl{i}.
\]
Note the intersection is algebraically transverse inside $\bd[\pf,\pfQ]$ (see \Cref{appendix-transverse} for the definition of algebraically transverse).

If $\pfQ$ is non-empty, then $\bdexl{i}$ is non-reduced and $\bdexl[\pf+\pfQ]{i}$ is the  induced reduced substack. Furthermore, $\bdexl{i}$ has geometric multiplicity $\ell(\pfQ)$ along  $\bdexl[\pf+\pfQ]{i}$ and $\bdexl[\pf,\pfQ]{i}$  is regularly embedded in $\bd[\pf,\pfQ]$ with normal bundle $\nb{\bdexl{i}}{\bd[\pf,\pfQ]}$.
We thus have 
\begin{gather*} 
[\bdexl{i}] = \ell(\pfQ)[\bdexl[\pf+\pfQ]{i}]\in \CH(\bd),\\c(\nb{\bdexl{i}}{\bd[\pf,\pfQ]})\cap [\bdexl{i}]  = \ell(\pf[Q]) c(\nb{\bdexl[\pf+\pfQ]{i}}{\bd[\pf+\pfQ]})\cap [\bdexl[\pf+\pfQ]{i}]\in \CH(\bdexl[\pf+\pfQ]{i}).
\end{gather*}

For later use we notice that the collection of substacks $\bdexl{i}$ behaves well with respect to intersection.
\begin{proposition}\label{prop:closedintersection}
Let $\pf$ be a profile and $\pfQ,\pfR\in \PF[(\pf,i)]$.
Suppose \[
I\subseteq \{i,\ldots,i-L(\pfQ)+1\}, J\subseteq \{i,\ldots, i-L(\pfR)+1\}
\] are subsets. Then
\[
\bdexl{I}\cap\bdexl[\pf,\pfR]{J} = \bdexl[\pf,\pfQ+\pfR]{K}
\]
for some subset $K\subseteq\{i,\ldots,i-L(\pfQ+\pfR)+1\}$.
\end{proposition}

\begin{proof}
We have the containment $\bdexl{I}\cap\bdexl[\pf,\pfR]{J}\subseteq  \bd[\pf,\pfQ+\pfR]$.
At a point  $(X,\omega)$  in $\bd[\pf,\pfQ+\pfR]$ the local defining equations for $\bdexl{I}\cap\bdexl[\pf,\pfR]{J}\subseteq  \bd[\pf,\pfQ+\pfR]$  are given by log periods
$\alpha=0$
for $\vb_{\pf+\pfQ+\pfR}$ 
 over a basis of cycles corresponding to levels that undegenerate to a level in $I$ under $\pf+\pfQ+\pfR\rightsquigarrow \pf+\pfQ$ or to a level in $J$ under $\pf+\pfQ+\pfR\rightsquigarrow\pf+\pfR$.
If we denote $K$ the set of all these levels of $\pf+\pfQ+\pfR$, then
\[
\bdexl{I}\cap\bdexl[\pf,\pfR]{J} = \bdexl[\pf,\pfQ+\pfR]{K}.
\]
\end{proof}

\begin{proposition}\label{prop:zerolocus}
The zero locus $Z(\cl{\ev}_{\pfP,i}) \subseteq\bd$ is the scheme-theoretic union of 
\[
Z(\cl{\ev}_{\pfP,i}) =  \bdexl[\pfP]{i}\cup \bigcup_{\lG\in \PFi[(\pfP,i)]{1}} \bdexl[\pfP,\lG]{i},
\]
where we recall that $\PF[{}^1(\pfP,i)]$ is the collection of non-horizontal two-level graphs  $\lG$  such that the degeneration $\pfP\degen\pf+\lG$ is obtained by splitting the $i$-th level.

\end{proposition}

\begin{proof}

Let $(X,\omega)\in\bd[\lGh]^{\circ}$ for some level graph $\lGh$.
%We set $\calW:=\vb_{\pf}^{[i]}$ for the rest of the proof.
On $\lGh$ the $i$-th level of $\pf$ has split into  levels $j,\ldots, j-k$ for some choice of $j,k$.

We first deal with the case that $\lGh$ has no horizontal nodes.
In order to compute the evaluation section in local coordinates we choose a tuple 
  $\alpha= (\alpha_j,\ldots,\alpha_{j-k})$, where $\alpha_l$ is a basis for log periods  for the bundle $\vb_{\lGh}^{[l]}$.
In \cref{eq-rescaled-evaluation} we computed the local equations for $\cl{\ev}_{\pf,i}$, thus the ideal $\calI$ of the zero locus of $\cl{\ev}_{\pf,i}$ evaluation section is given by 
\[
\begin{split}
&\calI = (\top_{j}\alpha_j,\ldots,\top_{j-k}\alpha_{j-k}),
\end{split}
\]
where we set  $\top_{l}:= \dfrac{\rescl[l]}{\rescl[j]} = \prod_{s=j-k}^{j} t_s^{\ell_s}$.
The ideal can be rewritten as an intersection of ideals
\[
\calI = (\alpha_j,t_{j-1}^{\lexp[j-1]})\cap (\alpha_j,\alpha_{j-1}, t_{j-2}^{\lexp[j-2]}) \cap\ldots\cap (\alpha_j,\ldots,\alpha_{j-k-1},t_{j-k}^{\lexp[j-k]})\cap (\alpha_j,\ldots,\alpha_{j-k}).
\]

The last ideal  $(\alpha_j,\ldots,\alpha_{j-k})$ is the ideal of $\bdexl{i}$. The remaining ideals are the defining ideals of $\bdexl[\pf,\lG]{i}$ for some $\lG\in\BIC$, where $\lG$ is  constructed as follows. Fix some ideal
$\calJ=(\alpha_j,\ldots,\alpha_{j-l+1},t_{j-l}^{\lexp[j-l]})$. 
The two-level graph $\lG$ is obtained by only keeping the level passage between the level $j-l+1$ and $j-l$ and contracting all other levels, in other words $\lG=\delta_{l-j+1}(\lGh)$.
By comparing $\calJ$ with the local defining equations for $\bdexl[\pf,\lG]{i}\subseteq \bd[\pf]$ from \cref{eq:bdexl}, we see that the ideals coincide.

We now address the remaining case where $\lGh$ has horizontal nodes.
As before let $I=\{j,\ldots,j-k\}$ be the levels undegenerating to level $i$.
First note that if $\lGh$ has some horizontal node  in a level $l$ not contained in $I$, we can smooth out level $j$ and still stay inside the zero locus of $\calI$.
We also claim that $\lGh$ cannot have a horizontal node $e$ in level $j$. To see this note that the vanishing cycle $\lambda_e$ has a non-zero period and is contained in $\vb_{\pf}^{[i]}$, since we assumed $\vb_{\pf}^{[i]}$ to be a globally defined vector bundle. Thus the evaluation morphism cannot vanish on the $j$-th level.
The remaining case is that $\lGh$ has a horizontal node $e$ contained in a level in $I\setminus \{j\}$. In this case all periods appearing in $\cl{\ev}_i$ involving the horizontal node $e$ are multiplied by some rescaling factor $t_l^{\lexp[k]}$.
Thus we can smooth out the horizontal node while staying inside $Z(\cl{\ev}_{\pf}^{[i]})$. Hence we reduced to the case of non-horizontal level graphs.
\end{proof}

\subsection{Residue subspaces}
We can  now give a quick proof of \Cref{prop:residue-subspaces}.
Consider the residue map
\[
\Res:\calO_{\PSt}(-1)\to \calO_{\PSt}^r,
\]
where $r$ is the number of poles. Let $W\subseteq \calO^r$ be a subbundle.
A residue subspace $Z\subseteq\PMSDS$ is the closure of  the zero locus of the composition $\calO_{\PSt}(-1)\to \calO_{\PSt}^r\to \calO_{\PSt}^r/W$.
Thus a generic point in $Z$ consists of differentials whose residues are contained in $W$.

Dually we can also describe a residue subspace using vanishing of periods. For this let $V=(\calO^r/W)^*$ be such that $\ANN(V)= W$.
Recall the trivial local system $\calH_{\res}$ which is generated by small loops around the marked poles and let $\calV\subseteq\calH_{\res}$ be the sub local system corresponding to $V$. Here any element in $V$ defines a linear equation among residues, which we consider as an element of the dual of cohomology and thus as a relative homology class.
Then $\calV$ is a globally defined local system on the stratum and $Z=\cl{\Ann{\PSt}}$.
We say $Z$ is a {\em residue subspace defined by } $V$.

Let $k=\dim V$ and choose a complete flag
\[
\{0\}=V_0 \subsetneq V_1\subsetneq\ldots\subsetneq V_k=V.
\]
This induces a flag of subbundles $\vb_{\bullet}:\vb_0\subsetneq \vb_1\subsetneq\ldots\subsetneq\vb_k=\vb$.

\begin{proposition}\label{prop:residue-closed} Let $Z$ be a residue subspace defined by a $k$-dimensional linear subspace $V\subseteq \CC^r$.
The class of $Z$ in the Chow ring of $\PMSDS$ is given by
\[
[Z] = (-1)^k\prod_{i=1}^{k}\left( \taut+\sum_{\lG\in W_i(\vb_{\bullet})} \ell_{\lG}\bd[\lG]\right)\in \trd[k]{\PMSDS},
\]
where 
 $W_i(\vb_{\bullet})\subseteq\BIC$ is the collection of all non-horizontal two level graphs $\lG$ such that the restrictions $ (\vb_i)_{\lG}^{[0]} = (\vb_{i-1})_{\lG}^{[0]}$ agree.
 
\end{proposition}

\begin{proof}
We define a chain of subspaces by letting  $Z_i=\cl{\Ann[\vb_i]{\PSt}}$. In particular we have $\PSt= Z_0, Z=Z_r$. Note that $Z_i\subseteq Z_{i-1}$ is a divisor, since it equals the closure of the annihilator of $\left(\vb_{i}/\vb_{i-1}\right)_{|Z_{i-1}\cap\PSt}$.

The argument  in \cite[Prop. 8.3]{CMZ} (which follows closely \cite[Prop 7.6]{sauvaget-cohomology}) is only stated for linear subspaces $V\subseteq \CC^r$ of a special form induced by global residue conditions but can be modified to show 
\[
[Z_i] = -(\taut+\sum_{\lG\in W_{i-1}} \ell_{\lG}\bd[\lG])\in\CH[1](Z_{i-1}).
\]
To modify the argument we note that the evaluation section $\vb_i/\vb_{i-1}\otimes\calO(-1)\to\calO$ vanishes on $Z_i$ with multiplicity one, which can be seen in period coordinates. Furthermore it vanishes identically on a boundary divisor $\bd[\lG] \cap Z_{i-1}$ if and only if all the residues in $V_i$ are on the lower level of $\bd[\lG]$, at least modulo $ (\vb_{i-1})_{\lG}^{[0]}$. On the lower level the universal family of multi-scale differentials is given by the differential $t^{\ell_{\lG}}\omega$ and thus the residue of a marked point in the bottom level vanishes with multiplicity $\ell_{\lG}$. The proof now follows by recursively using push-pull.
\end{proof}

\section{The arrangement of exact differentials}
For the rest of this section we fix a globally defined local system $\calV$ on a generalized stratum $B$ as well as a profile $\pf$ and a level $i$. Our goal is  to compute the class of $\bdexl[\pf]{i}$.

We order all non-horizontal  two-level graphs $\lG_1,\ldots,\lG_n\in \PF[{}^1(\pf,i)]$ in a way that respects the partial ordering $\succ$ on $\BIC$. Recall the definition of $\PF[{}^1(\pf,i)]$ from \Cref{prop:zerolocus}. We then blow $B$ iteratively up  in $\bdexl[\pf,\lG_1]{i},\ldots, \bdexl[\pf,\lG_n]{i}$.
On the final blowup $\ptf{Z}$ there will be a vector bundle $\calE$ and a regular section whose zero locus is the proper transform $\ptf{\bdexl[\pf]{i}}$.
By pushing forward along $\tilde{Z}\to B$ we then obtain a formula for the class and the Chern class of the normal bundle of $\bdexl[\pf]{i}$ in terms of the classes and normal bundles of the blowup centers $\bdexl[\pf,\lG_j]{i}$. We then repeat the procedure for  $\bdexl[\pf,\lG_j]{i}$ and the globally defined vector bundle $\vb_{\pf+\lG_j}^{[i]}$.
At each step the codimension of the underlying boundary stratum $\bd[\pf]$ increases and therefore the procedure has to terminate at some point, allowing us to compute the class $\bdexl[\pf]{i}$ recursively.

For the remainder of the section we will freely make use of the definitions and statements from the appendix.

\subsection{The blowup construction}
We start with the explicit construction of $\ptf{Z}$. We now write $Z=B$.
For brevity set $W_k := \bdexl[\pf,\lG_k]{i}$ and $A:= \bdexl[\pf]{i}$, $\ev=\cl{\ev}_{\pf}^{[i]}$.
We then define the first blowup $\ptf{Z}_1:=\Bl_{W_1} Z$. Afterwards we let $\ptf{W}^{(1)}_i$ be the proper transform of $W_i$ for $i>1$ and $\ptf{W}_1^{(1)}$ and the exceptional divisor. Define the second blowup $\ptf{Z}_2:= \Bl_{\ptf{W}_2^{(1)}} \ptf{Z}_1$.
 We then proceed inductively. If we have already constructed $\ptf{Z}_i$, we  let  $\ptf{W}_{j}^{(i)}$ be the iterated proper transform of $W_j$ for $j>i$ and $\ptf{j}^{(i)}, j\leq i$ the pullback of the exceptional divisor. 
 We then define the $(i+1)$-th blowup by
\[ \ptf{Z}_{i+1}:= \Bl_{\ptf{W}_{i+1}^{(i)}} \ptf{Z}_i.\] We write $\ptf{Z}:=\ptf{Z}_n$ for the final blowup.

Recall that on $Z$ we have the vector bundle $\calE:=\left(\vb_{\pf}^{[i]}\right)^*\otimes \calO_Z^{[i]}(1)$ and the evaluation morphism $\ev:\calE^* \to\calO_Z$, which we consider a global section of $\calE$.
We can pullback $\calE$ and $\ev$ to all partial blowups $\ptf{Z}_i$ and usually omit the pullbacks in our notation. Since $\ev$ vanishes on $A_{\pf,\lG_1}$, the pullback section vanishes with some multiplicity $m(\lG_1)$ on the exceptional divisor $\ptf{W}_1^{(1)}$. Dividing $\ev$ by the defining equation for $\ptf{W}_1^{(1)}$, we obtain a section $ 
\ptf{\ev}_1$ of the vector bundle $ \calE \otimes \calO(\ptf{W}_1^{(1)})^{\otimes -m(\lG_1)}$.

We now claim that 
\[
Z(\ptf{\ev}_1) = \ptf{A}^{(1)} \cup_{j\geq 2} \ptf{W}_j^{(1)}.
\]
In particular we removed one of the exceptional components.

This will be proven as part of \Cref{prop:finalblowup}.
From here we proceed inductively.
We now blow $\ptf{W}_2^{(1)}$ up, divide the pullback of $\ptf{\ev}_1$ by the defining equation for the exceptional divisor and obtain a section $\ptf{\ev}_2$ of
$\calE \otimes  \calO(\ptf{W}_1^{(2)})^{\otimes -m(\lG_1)}
\otimes  \calO(\ptf{W}_2^{(2)})^{\otimes -m(\lG_2)}$.
The vanishing locus of $\ptf{\ev}_2$ will then be 
\[
Z(\ptf{\ev}_2) = \ptf{A}^{(2)} \cup_{j\geq 3} \ptf{W}_{j}^{(2)}.
\]
After the final blowup we will thus have constructed a section $\ptf{\ev} = \ptf{\ev}_n$ of 
\[
\calE \otimes  \bigotimes_{i=1}^n \calO(\ptf{W}_i)^{\otimes -m(\lG_i)}
\]
whose zero locus is 
\[
Z(\ptf{\ev}) = \ptf{A}
\]
the proper transform of $A= \bdexl[\pf]{i}$.
Since $\codim_{\ptf{Z}} \ptf{A} = \rank \calE$ has the expected codimension, the section $\ptf{\ev}$ is regular and thus we can compute both the class $[\ptf{A}]$ and the Chern classes of the normal bundle of $\ptf{A}$.
All this information is summarized in the \CHP $\chp{\ptf{A}}{\ptf{Z}}$ (see \Cref{sec:chp}).

\begin{proposition}
\label{prop:finalblowup}
On the final blowup $\ptf{Z}$ we have 
\[
\chp{\ptf{A}}{\ptf{Z}} = c( \calO^r  \otimes\calO(1)\otimes \bigotimes_{i=1}^n \calO(-\ptf{W}_i)),
\]
where $r$ is the rank of $\calE$.

\end{proposition}

\begin{proof}
First note that $\vb_{\pf}^{[i]}$ has vanishing Chern classes by \Cref{prop:vanishing} so for the purpose of Chern class computation we replace it by the trivial bundle.
We now address the multiplicity $m(\lG_k)$.
It can be computed as the lowest degree term of $\ptf{\ev}^{(k-1)}$ along a generic point of $\ptf{W}_k^{(k-1)}$. Since $\ptf{W}_k^{(k-1)}$ and $W_k$ are birational it suffices to compute the multiplicity of $\ev$ along $W_k$.
A generic point of $W_k$ has profile $\pf+\lG_k$.
There are now two possibilities. Either there exists some period of $\vb$ in level $i$, in which case the multiplicity along that period is $1$. Therefore the evaluation section vanishes with multiplicity $1$ in this case.
In the remaining case all periods of $\vb$ are in lower level, thus the evaluation section can be written as $t_i^{\ell_k}\alpha$, where $\alpha$ denotes the log periods of $\vb$ in the lower level and $\ell_k= \ell(\lG_k)$. Generically some lower level period is non-zero. Thus the multiplicity is $1$, since
 we are working with the non-reduced space $\bdexl[\pf,\lG_k]{i}$, where  in the blowup $t_i^{\ell(\lG_i)}$ is a uniformizer.

It remains to show the claim that at each step
\[
Z(\ptf{\ev}_k) = \ptf{A}^{(k)} \cup \bigcup_{j=k+1}^n \ptf{A}_{\pfP,\lG_j}^{(k)}.
\]
We argue by induction. The case $k=0$ follows from \Cref{prop:zerolocus}.
For the induction step we assume the statement has been proven for some $k$. We claim that on $\ptf{Z}_k$  local equations for $\ptf{W}_{k+1}^{(k)}=\ptf{A}_{\pfP,\lG_{k+1}}^{(k)}$ are given by a regular sequence 
\[
(\alpha,T_k)
\]
where $\alpha$ is a tuple corresponding to the periods specializing to level $i$ and $T_k$ corresponds to the transversal coordinate $t_{i+1}^{\ell(\lG_{k+1})}$.
This can be seen by induction, using the observation that local equations for the proper transform are obtained by replacing some terms in the regular sequences by corresponding projective coordinates; see \Cref{sec:equationproper} for a description in local coordinates.
Similarly the evaluation section can now be written in the form
\[
\ptf{\ev}_k=(\alpha,T_k\cdot \beta),
\]
where $\beta$ is a tuple such that $(\alpha,\beta,T_k)$ is a regular sequence.

 Thus on the blowup, after dividing by the exceptional divisor, the zero locus of $\ptf{\ev}_{k+1}$ is given by $(\ptf{\alpha},\beta)=0$, where $\ptf{\alpha}$ denotes the corresponding projective coordinates. In particular there is no component of the zero substack of the evaluation section that is supported entirely on the new exceptional divisor $\ptf{W}_{k+1}^{(k+1)}$ and the remaining components are of the form $\ptf{A}_{\pfP,\lG_j}^{(k)}$.
In order to reduce the amount of bookkeeping we discuss the details only in  the following case. Suppose $(X,\omega)\in W_2$ with underlying   $3$-level graph $\lGh$. Additionally suppose the profile of $\lGh$ is $\lG_1,\lG_2$ and we have $W_i = \bdexl[\pf,\lG_1]{i}$.  Let  $\alpha,\beta,\gamma$ be a basis of the restriction $\vb^{[i]}_{\lGh}$ for levels $0, -1$ and $-2$ respectively and let $t_1$ and $t_2$ be the level parameters for level $-1$ and $-2$, respectively. For simplicity we assume $i=0$ and $\ell(\lG_1)=\ell(\lG_2)=2$.
The ideal sheaves for $W_1$ and $W_2$ are locally generated by $\calI(W_1)= (\alpha,t_1), \, \calI(W_2)= (\alpha,\beta,t_2)$, while the evaluation section is given by
\[
\ev = (\alpha,t_1\beta,t_1t_2\gamma).
\]

On the exceptional divisor we have to introduce projective coordinates $\alpha'$ and $t_1'$.
At a point supported on the exceptional divisor of the blowup $\ptf{Z}_1$ the equation for the proper transform of $W_2$ are given by a regular sequence 
$(\alpha',\beta,t_2)$ where $\alpha'$ denotes the projective coordinates corresponding to $\alpha'$.
On the chart where $\alpha'_j\neq 0$, the evaluation section $\ptf{\ev}_{(1)}$ is given by
\[
\ptf{\ev}^{(1)} = (\hat{\alpha},t_1'\beta,t_1't_2\gamma),
\]
where $\hat{\alpha} = (\alpha'_1,\ldots,\alpha'_{j-1},1,\alpha'_{j+1},\ldots,\alpha'_d)$. Thus $\ptf{\ev}_{(1)}$ is non-vanishing.
On the chart where $t_1'\neq 0$, we have
\[
\ptf{\ev}_{(1)} = (\alpha',\beta,t_2\gamma).
\]
Thus the zero locus is the union of $Z(\alpha', \beta, t_2) = \ptf{W}_2^{(1)}$ and $Z(\alpha', \beta, \gamma) = \ptf{A}^{(1)}$.

\end{proof}

We now have all the ingredients to finish the proof of the main theorem.
\begin{proof}[Proof of \Cref{thm:main}]
In our new language we are trying to compute the \CHP of $\bdexl[\star]{0}$, where $\star$ denotes the unique level graph with a single vertex.

We will show more generally that the \CHP for $\bdexl[\pf]{I}$ in $\bd$ can be computed explicitly and lies in the divisorial tautological ring $\divR(\bd)$.
As we have already remarked the intersection
\[
\bdexl[\pf]{I} = \cap_{i\in I} \bdexl[\pf]{i}\subseteq\bd,
\]
is algebraically transverse. Thus it suffices to compute the \CHP of $\bdexl[\pf]{i}$ for all profiles and all levels $i$.

\begin{claim}
Suppose that we have already shown $\bdexl[\pfP+\pfQ]{j}\in\divR(\bd[\pf+\pfQ])$ for any $\pfQ\in\PF[(\pf,i)]$ and  any level $j$ that undegenerates to the $i$-th level.
Then $\bdexl[\pf]{i}\in\divR(\bd[\pf])$ and can be  computed explicitly.
\end{claim}

Before we prove the  claim we explain how it finishes the proof of the theorem. For a profile $\pf$ with no degenerations that split level $i$, it follows from \Cref{prop:zerolocus}
that
\[
Z(\cl{\ev}_{\pf,i}) = \bdexl[\pf]{i}
\]
and thus $\cl{\ev}_{\pf,i}$ is a regular section. In particular the \CHP of $\bdexl[\pf]{i}$ is given by the \CHP of $(\vb_{\pf}^{[i]})^*\otimes\calO_{\bd}^{[i]}(1)$ and thus $\chp{\bdexl[\pf]{i}}{\bd[\pf]} = c(\calO_{\bd}^r\otimes \calO_{\bd}^{[i]}(1))\in\divR(B)$,
where $r=\rank \vb_{\pf}^{[i]}$.
The general case follows now from the claim by induction on the maximal length of a profile in $\PF[(\pf,i)]$. Since the length of profiles is bounded the process terminates.

We now prove the claim.
Our goal is to apply \Cref{prop:arrangement} with the system of regular embeddings given by
\[
\calS = \{\bd\}\cup\{ \bdexl[\pf, \pfQ]{I}\,:\, \pfQ \in \PF[(\pf,i)], L(\pfQ)\geq 2\text{ and any level in $I$ undegenerates to level $i$}\},
\]
the building set $B=  \{ \bdexl[\pf, \lG]{i}\,:\, \lG \in \PF[{}^1(\pf,i)]\}$
and $A= \bdexl[\pf]{i}, \Rtr = \divR$. To see that $\calS$ is indeed a system of regular embeddings, we need to check that $\calS$ is closed under finite intersections and if $\bdexl[\pf,\pfQ]{I}\subseteq\bdexl[\pf,\pfR]{J}$ is a containment, then the inclusion is a regular embedding. Closedness under finite intersection follows from \Cref{prop:closedintersection}.
By inspecting the local equations for  $\bdexl[\pf,\pfQ]{I}$ and $\bdexl[\pf,\pfR]{J}$ inside $B$, we see that $\bdexl[\pf,\pfQ]{I}$ is defined by the additional vanishing of log periods in the levels in $I$ that do not undegenerate to $J$ and some additional level scaling parameters $t_i^{\ell_i}$, which are part of a regular sequence and thus showing that inclusion is a regular embedding. We also need to verify that all $\bdexl[\pf,\pfQ]{I}$ are Alexander stacks. By \Cref{prop:Alexander}, 2) it suffices to show the claim for the reduced substack $\bdexl[\pf+\pfQ]{I}$. But $\bdexl[\pf+\pfQ]{I}$ is smooth, since the local equations are part of a system of parameters for $B$ and any smooth \DM stack is Alexander, see \Cref{prop:Alexander}, 3).

We now address the \CHPs. First, we show that for any $\bdexl[\pf+ \pfQ]{I}\in\calS$ the \CHP $\chp{\bdexl[\pf+ \pfQ]{I}}{\bd[\pf]}$ is defined over $\divR$.
By the  assumption of the claim we know $\chp{\bdexl[\pf+ \pfQ]{I}}{\bd[\pf+\pfQ]} = \prod_{j\in I} \chp{\bdexl[\pf+ \pfQ]{j}}{\bd[\pf+\pfQ]} $ is defined over $\divR$ for all $\bdexl[\pf+\pfQ]{I}\in\calS$ and hence
\[
\chp{\bdexl[\pf, \pfQ]{I}}{\bd[\pf,\pfQ]} = \cdot \chp{\bdexl[\pf+ \pfQ]{I}}{\bd[\pf+\pfQ]} \in \trd{\bd[\pf+\pfQ]}.
\]

On the other hand $\chp{\bd[\pf+\pfQ]}{\bd}\in\divR(\bd[\pf])$. We then apply \Cref{prop:chern} to the inclusions $\bdexl[\pf,\pfQ]{I}\subseteq\bd[\pf,\pfQ]\subseteq\bd$.

It only remains to show that for any inclusion $\bdexl[\pf, \pfQ]{I}\subseteq \bdexl[\pf, \pfS]{K}$ with  $\bdexl[\pf, \pfQ]{I}\in\calS, \bdexl[\pf, \pfS]{K}\in\calS\cup\{\bdexl[\pf]{i}\}$ the \CHP $\chp{\bdexl[\pf, \pfQ]{I}}{\bdexl[\pf, \pfS]{K}}$ is defined over $\divR$.

We can write $\pfQ= \pfR + \pfS$ for some profile $\pfR$ and then
\[ \bdexl[\pf,\pfS]{K} \cap \bd[\pf,\pfQ]= \bdexl[\pf,\pfS]{K} \cap \bd[\pf,\pfR]= 
 \bdexl[\pf,\pfQ]{J} 
\] for some set $J\subseteq I$.
Since $\pfR$ and $\pfS$ are disjoint, the intersection $ \bdexl[\pf,\pfQ]{J} =\bdexl[\pf,\pfS]{K} \cap \bd[\pf,\pfR]$ is algebraically transverse inside $\bd[\pf]$. In particular we see
\begin{equation}
\label{eq:normal2}\chp{\bdexl[\pf,\pfQ]{J}}{\bdexl[\pf,\pfS]{K}} = \chp{\bd[\pf,\pfR]}{\bd[\pf]} = \ell(\pfR) \chp{\bd[\pf+\pfR]}{\bd}\in \Rtr(\bdexl[\pf,\pfS]{K}).
\end{equation}

Additionally,
\[
\bdexl[\pf,\pfQ]{I} = \bdexl[\pf,\pfQ]{J} \cap \bigcap_{j \in I \setminus J}\bdexl[\pf,\pfQ]{j}\subseteq \bd[\pf,\pfQ] 
\]
is an algebraically transverse intersection and thus
\begin{equation}
\label{eq:normal1}
\chp{\bdexl[\pf,\pfQ]{I} }{\bdexl[\pf,\pfQ]{J} } =\prod_{i\in I\setminus J} \chp{\bdexl[\pf,\pfQ]{i}}{\bd[\pf,\pfQ]} = \prod_{i\in I\setminus J} \chp{\bdexl[\pf+\pfQ]{i}}{\bd[\pf+\pfQ]}  \in \Rtr(\bdexl[\pf,\pfQ]{J}),
\end{equation}
omitting evident pullbacks.

The situation is summarized in the following diagram of Cartesian squares.
\[
\begin{tikzcd}
\bdexl[\pf,\pfQ]{I}\ar[r]\ar[d]& \bdexl[\pf,\pfQ]{J} \ar[r]\ar[d] & \bdexl[\pf,\pfS]{K}\ar[d]\\
 \cap_{j\in I/J} \bdexl[\pf,\pfQ]{j} \ar[r]&\bd[\pf,\pfQ] \ar[r] \ar[d]& \bd[\pf,\pfS]\ar[d]\\
& \bd[\pf,\pfR] \ar[r]& \bd[\pf]
\end{tikzcd}
\]

By  \cref{eq:normal1,eq:normal2}, it now follows
from \Cref{prop:chern} (2) that
$\chp{\bdexl[\pf, \pfQ]{I}}{\bdexl[\pf, \pfS]{K}}\in \trd{\bdexl[\pf, \pfS]{K}}$.

This finishes the proof of the induction step.
\end{proof}

\section{Low genus formulas and crosschecks}\label{section-crosschecks}
The iterative procedure for  the computation of the class of $[\cl{\Exc}]$ has been implemented in \verb|SAGE|, building on the packages \verb|admcycles| and 
\verb|diffstrata|. The code is available on the authors website.
For the algorithm we need to list all two-level graphs as well as the partial ordering between them.
The number of  level graphs grows exponentially and thus slows down the computation.
In this section we explain a few cases where we can avoid the complicated recursion using blowups and obtain more explicit formulas.

\subsection*{Imposing ramification}

If one has already computed the class $[\cl{\Excmu}]$ inside the Chow ring of $\PMSDS$, then imposing additional ramification imposes divisorial conditions.
In particular we can obtain a closed formula for the class of $[\cl{\Exc}]\in \CH(\cl{\Excmu})$ as follows.

Let $\ramp$ be a ramification profile. Recall the sublocal system $\calH_{\ramp}\subseteq \calH$ which is spanned by the absolute homology bundle $\calH_{\abs}$ whose fiber over a point $(X,\omega)$ is $H_1(X - P(\omega))$, together with all relative cycles whose end points are contained in a fiber corresponding to a partition $\lambda_i$ of $\ramp$.
Restricted to $\Excmu$ the local systems $\calH_{\ramp}/ \calH_{\abs}$ has trivial monodromy and we can thus choose a flag of subbundles
\[
\calU_{\bullet}: \calH_{\abs} = \calU_0 \subsetneq \ldots \subsetneq \calU_l = \calH_{\ramp}.
\]
In fact we can choose the flag such that $\calU_i/\calU_{i-1}$ is generated by a single cycle with endpoints in the same fiber.
We now let $W_i(\calU_{\bullet})\subseteq \BIC$ be the collection of two level graph such that for the top level restrictions we have  $\calU_{i,\lG}^{[0]} = \calU_{i-1,\lG}^{[0]}$. Recall the definition of the top level restriction from \Cref{def:rlg}. In other words $W_i(\calU_{\bullet})$ consists of the level graphs where the generator of $\calU_{i}/\calU_{i-1}$ is on the bottom level modulo cycles in $\calU_{i-1}$. We stress that $W_i(\calU_{\bullet})$ depends on the choice of flag $\calU_{\bullet}$.
Exactly as in the proof of \Cref{prop:residue-closed} we obtain 
\begin{proposition}\label{prop-torsion}
Let $\ramp$ be a ramification profile. The class of the closure of $\Exc$ inside the Chow ring of $\cl{\Excmu}$ is given by
\[
[\cl{\Exc}] =  (-1)^l\prod_{i=1}^{l}\left( \taut+\sum_{\lG\in W_i(\calU_{\bullet})} \ell_{\lG}\bd[\lG]\right)\in \trd[l]{\cl{\Excmu}},
\]
where $l= \codim_{\Excmu} \Exc$.
\end{proposition}

In particular we 
have
\[
[\cl{\Exc}] =  (-1)^l\prod_{i=1}^{l}\left( \taut+\sum_{\lG\in W_i(\calU_{\bullet})} \ell_{\lG}\bd[\lG]\right)\cdot [\cl{\Excmu}]\in \CH(\PMSDS).
\]

\subsection*{Genus zero and realizability of branched covers of the sphere}

Every stratum of  $\ramp$-exact differentials can be decomposed as
\[
\cl{\Exc} \subseteq \cl{\Excmu} \subseteq \PMSDS^{\res}\subseteq \PMSDS,
\]
where $\PMSDS^{\res}$ is the closure of the stratum of residueless differentials.
We have closed formulas for the class of $\Exc$ inside $\Excmu$ and for $\PMSDS^{\res}\subseteq \PMSDS$. Thus if we let $\iota: \PMSDS^{\res}\subseteq \PMSDS$ be the inclusion and $\alpha \in \divR(\PMSDS)$ be a tautological class such that $\iota^*\alpha = [\cl{\Excmu}]\in \divR\PMSDS^{\res}$, then
\[
[\cl{\Exc}] =  (-1)^{k+l}\prod_{i=1}^{l}\left( \taut+\sum_{\lG\in W_i(\calU_{\bullet})} \ell_{\lG}\bd[\lG]\right)  \prod_{j=1}^{k}\left( \taut+\sum_{\lG\in W_j(\vb_{\bullet})} \ell_{\lG}\bd[\lG]\right)\cdot \alpha\in\divR(\PMSDS),
\]
where we combined \Cref{prop:residue-closed} and \Cref{prop-torsion}.

For general $g$, we can only use the recursive blowup procedure to compute $\alpha$ but for $g=0,1$ the codimension is small enough that we can obtain closed formulas.

For genus zero, every residueless differential is exact and thus $\PMSDS^{\res} = \cl{\Excmu}$, i.e., we have $\alpha=1$.

\begin{proposition}\label{prop-g-0}
Let $\Exc$ be a stratum of $\ramp$-exact differentials in genus $g=0$. Then
\[
[\cl{\Exc}] =  (-1)^{k+l}\prod_{i=1}^{l}\left( \taut+\sum_{\lG\in W_i(\calU_{\bullet})} \ell_{\lG}\bd[\lG]\right)  \prod_{j=1}^{k}\left( \taut+\sum_{\lG\in W_j(\vb_{\bullet})} \ell_{\lG}\bd[\lG]\right)\in \trd[k+l]{\PMSDS},
\]
where \begin{itemize}

\item $k = \codim_{\Excmu} \Exc,\, l = \codim_{\PMSDS} \Excmu$,
\item $\vb_{\bullet}$ is a complete flag of subbundles of the residue local system $\calH_{\res}=\calH_{\abs}$,
\item $\calU_{\bullet}$ is a complete flag of subbundles of $(\calH_{\ramp}/ \calH_{\abs})_{|\PMSDS^{\res}}$ 
\end{itemize}
\end{proposition}

It is still an open question which ramification profiles $\ramp$ can be realized as branched covers of $\PP^1 \to \PP^1$, see for example \cite{hurwitz-existence} for the current state. A necessary condition is that $\ramp$ satisfies the conditions  of the Riemann-Hurwitz theorem but in general this is not sufficient.
For example in degree $d=4$, the ramification profile $\ramp = (3,1),(2,2),(2,2)$ cannot be realized by a branched covering.
We can verify this by computing that the class of the corresponding stratum of $\ramp$-exact differentials has zero class in $\CH(\PMSDS)$. We do not have an efficient way of checking that a class is zero in the Chow ring of $\PMSDS$ and instead it is easier to pushforward along  $\rho:\PMSDS \to \Mgnbar[0,n]$.
Using the \verb|SAGE| package \verb|admcycles| and \Cref{prop-g-0} we can verify that indeed for $\ramp = (3,1),(2,2),(2,2)$ the class $[\cl{\Exc}]$ is zero in $\CH(\Mgnbar[0,6])$.
Similarly we can check for the remaining degree $d=4$ ramification profiles that the class is non-zero.

Since for a projective variety a subvariety is non-empty if and only if its class in the Chow ring is zero, we get the following criterion for when a branched cover can be realized.

\begin{corollary}Let $\ramp$ be a ramification profile in genus zero. Then $\ramp$ can be realized by a branched covering $\PP^1\to\PP^1$ if and only if the class 
\[
\rho_*\left[\prod_{i=1}^{l}\left( \taut+\sum_{\lG\in W_i(\calU_{\bullet})} \ell_{\lG}\bd[\lG]\right)  \prod_{j=1}^{k}\left( \taut+\sum_{\lG\in W_j(\vb_{\bullet})} \ell_{\lG}\bd[\lG]\right)\right] \neq 0\in \CH(\Mgnbar[0,n]).
\]

\end{corollary}

\subsection*{Genus one}

For genus $1$ the codimension of $\cl{\Excmu}$ inside $\PMSDS^{\res}$ is $2$. In particular the zero locus of the evaluation section has expected codimension, after we remove divisorial components.

We recall the local systems $\calH_{\abs}\supseteq \calH_{\res}$ of absolute homology and residues, respectively. We let $N_i\subseteq \BIC $ be all the two level graphs such that $\rank \calH_{\abs,\lG}^{[0]}-i = \calH_{\res,\lG}^{[0]}, i=0,1$. Here $ \calH_{\abs,\lG}^{[0]}, \calH_{\res,\lG}^{[0]}$ are the images of the specialization morphism of the Deligne extension of $\calH_{\abs}$ and $\calH_{\res}$, respectively. Then $N_i$ corresponds exactly to all two level graphs such that $A_{\lG}^{[0]}$ has codimension $1+i$ inside $\PMSDS$.

\begin{proposition}\label{prop-g-1} Let $\Exc$ be a stratum of $\ramp$-exact differentials in genus $g=1$. Then
\[
\begin{split}
[\cl{\Exc}] = &(-1)^{d}\prod_{i=1}^{l}\left( \taut+\sum_{\lG\in W_i(\calU_{\bullet})} \ell_{\lG}\bd[\lG]\right)  \prod_{j=1}^{k}\left( \taut+\sum_{\lG\in W_j(\vb_{\bullet})} \ell_{\lG}\bd[\lG]\right)\times \\
&\left [\left( \taut+\sum_{\lG\in N_0} \ell_{\lG}\bd[\lG]\right)^2 +
\sum_{\lG\in N_1}\ell_{\lG}\bd[\lG]\left( \taut+\sum_{\lG'\in N_1(\lG)} \ell_{\lG'}\bd[\lG']\right)
\right]\in \trd[*]{\PMSDS},
\end{split},
\]
where \begin{itemize}
\item $d = \codim_{\PMSDS} \Exc$,
\item $N_i\subseteq \BIC $ is the set of two level graphs such that $\rank \calH_{\abs,\lG}^{[0]}-i = \calH_{\res,\lG}^{[0]}$,
\item  $N_1(\lG)$ is the set of all two level graphs $\lG'$ with $\lG' \succ \lG'$ and 
\[\rank \calH_{\abs,\lG+\lG'}^{[0]}-1 = \rank\calH_{\res,\lG+\lG'}^{[0]}.\]
\end{itemize}
\end{proposition}

\begin{proof}
We recall that the zero locus of the evaluation section is supported on $A_{\star}^{[0]}$ and $A_{\star,\lG}^{[0]}$ for $\lG\in \BIC$. If $A_{\star,\lG}^{[0]}$ is non-empty, then the codimensions are either one or two.
In the first case  $A_{\star,\lG}^{[0]} = \bd[\lG]^{\bullet}$ and $\lG\in N_0$, where we recall that $\bd[\lG]^{\bullet}$ is an infinitesimal thickening of $\bd[\lG]$ with multiplicity $\ell_{\lG}$. The latter case corresponds to $\lG\in N_1$. 
In particular the zero locus of the evaluation section $\ev$, twisted by $\otimes_{\lG\in N_0}\calO(\bd[\lG])^{-\ell_{\lG}}$  has expected codimension and thus
\[
[\cl{\Excmu}] =[A_{\star}^{[0]}] = c_{2}\left( \calO(1) \otimes \bigotimes_{\lG\in N_0}\calO(\bd[\lG])^{-\ell_{\lG}}\otimes \calO^{\oplus 2} \right)- \sum_{\lG\in N_1} [A_{\star,\lG}^{[0]}]\in \CH(\PMSDS^{\res}).
\]
To compute the class of $A_{\star,\lG}^{[0]}$ we will instead compute  the class of  $[A_{\star,\lG}^{[0]}]\in\CH[1]( \PMSDS^{\res} \cap \bd[\lG]^{\bullet})$ and then use push-pull.
Similarly to before we have an evaluation section defined on  $\PMSDS^{\res} \cap \bd[\lG]^{\bullet}$ vanishing on $A_{\star,\lG}^{[0]}$ as well as \[
\bd[\lG']^{\bullet} \cap   \PMSDS^{\res} \cap \bd[\lG]^{\bullet}\text{ for }\lG\in N_1(\lG).
\]
And we have \[
[\bd[\lG']^{\bullet} \cap   \PMSDS^{\res} \cap \bd[\lG]^{\bullet}] = \ell_{\lG'}\bd[\lG']\in \CH[1]( \PMSDS^{\res} \cap \bd[\lG]^{\bullet}).
\]
\end{proof}

For example we can consider the  ramification profile $\ramp = ((2),(2),(2),(2))$, i.e., we consider 
\[
Z :=\{ (E,p,q,r,s)\in\Mgn[1,4]\,|\, \exists f: E\xrightarrow{2:1}\PP^1 \text{ ramified at } p,q,r,s\},
\]
or in other words elliptic curves with a full $2$-torsion package.
Pushing forward the class of $[\cl{\Exc}]$ forward we obtain
\[
\dfrac{75}{2}\kappa_3 - \dfrac{45}{4}\kappa_1\kappa_2 + \dfrac{3}{4}\kappa_1^3 + \dfrac{3}{2}\kappa_2\left(\psi_1 + \psi_2 + \psi_3 +\psi_4\right)\in \CH[3](\Mgnbar[1,4]),
\]
which agrees with the formula computed by \verb|admcycles|.

With additional work one can also obtain a formula for $g=2$ or even higher genus. The main difference is the for higher genus one has extraneous components which are neither divisorial nor of expected dimension and one cannot avoid the blowup procedure, which makes the formulas much more complicated.

\newpage
\appendix

\section{Intersection theory on iterated blowups}
\subsection*{Alexander stacks}

Let $X,Y$ be  separated Deligne-Mumford stacks over $ \CC$.
For a representable morphism $f:X\to Y$ we let $\CHop(X\to Y)$ be the  bivariant Chow groups.
An element $c\in\CHop(X\to Y)$ is an operator $c:\CH(Z)\to\CH(X \times_Y Z \to Z)$ for any morphism $f:Z\to Y$ from a scheme $Z$
, which is compatible with proper pushforward, flat pullback and Gysin homomorphisms for regular embeddings (see \cite[Chapter 17]{Fulton} for a precise definition).
As explained in Vistoli, an operational Chow class $c\in\CHop(X\to Y)$  also defines an operator $c\in\CHop(F\to X\times_Y F)$ for any representable morphism $F\to Y$ of \DM stacks.
Furthermore we have an evaluation map
\[
\ev_Y:\CHop(F\to Y)\to\CH(F), c \mapsto c \cap [Z].
\]
We denote by $\CHop(X):=\CHop(X\xrightarrow{id} X)$ the operational rational Chow ring. The ring structure on $\CHop(X)$ is induced from the product \[\CHop(X\to Y)\otimes\CHop(Y\to Z)\to \CHop(X\to Z)\] and is commutative.
%how class $c$ also defines an operator $\CH(F)\to\CH(Z)$ for every representable morphism $F\to Z$ of  \DM stacks.

We call a \DM stack $X$ an {\em Alexander stack} if $X$ is equidimensional and for every representable morphism $f:F\to X$ the evaluation  $\ev_X:\CHop(F\to X)\to \CH(F)$ is an isomorphism.
\begin{remark}
In \cite{VistoliAlexander} an additional commutative condition is required.
For applications in this paper we are working over a field of characteristic zero, so that every \DM stack has a resolution of singularities and thus the commutativity is automatically satisfied.
\end{remark}

We will use the following properties from \cite{VistoliIntersection}. Note that in (loc.cit.) everything is stated for schemes, but the proofs work analogously for \DM stacks. 
\begin{lemma}\label{prop:Alexander}
Let $X, Y, Z$ be \DM stacks and $f:X\to Y$ a representable morphism.
\begin{enumerate}
\item Suppose $f$ is smooth of constant fiber dimension. If $Y$ is an Alexander stack, so is $X$.
\item Suppose $f$ is a universal homeomorphism, then $X$ is an Alexander stack if and only if $Y$ is an Alexander stack. In particular, $X$ is Alexander if and only if the reduced stack $X^{red}$ is Alexander.
\item Every smooth \DM stack is an Alexander stack.
\end{enumerate}
\end{lemma}

\begin{proof}
(1) is \cite[Prop 2.2]{VistoliAlexander}, (2) is \cite[Prop 2.7]{VistoliAlexander} and
(3) is \cite[Prop 5.6]{VistoliIntersection}.

\end{proof}

The following is most likely well known to experts but we could not find a reference.

\begin{proposition}
Suppose that $X\subseteq Y$ are Alexander stacks and $X$ is regularly embedded in $Y$.
Then $\ptf{Y}:=\Bl_{Y} X$ is an Alexander stack.
\end{proposition}

\begin{proof}
Let $\ptf{X}$ be the exceptional divisor.
For any representable morphism $Y'\to Y$ we set $X'= X\times_Y Y', \, \ptf{X}' = X' \times_X \ptf{X}, \ptf{Y}'= Y'\times_Y \ptf{Y}$.

We have short exact sequences
\[
\begin{tikzcd}[sep=small]
0 \ar[r] & \CHop(X'\to X) \ar[r,"\ptf{\alpha}"] \ar[d,"\ev_X"] &\CHop(\ptf{X}'\to\ptf{X}) \oplus \CHop(Y'\to Y)\ar[r,"\ptf{\beta}"] \ar[d,"\ev_{\ptf{X}}\oplus \ev_Y"]& \CHop(\ptf{Y}'\to\ptf{Y}) \ar[r]\ar[d,"\ev_{\ptf{Y}}"] & 0\\
0 \ar[r] & \CH(X') \ar[r, "\alpha", swap]  &\CH(\ptf{X}') \oplus \CH(Y')\ar[r, "\beta", swap] & \CH(\ptf{Y}') \ar[r] & 0\\
\end{tikzcd}
\]
where the vertical maps are the evaluation maps.

The horizontal maps are given by
\[
\begin{split}
\ptf{\alpha}(c)& = (c_{d-1}(E)\cdot g^*c, i_*(c\cdot [i]), \quad \ptf{\beta}(r,s) = -j_*(r\cdot [j]) + f^*s,\\
\alpha(x) &= c_{d-1}(E)\cap g^*x, \quad \beta(\ptf{x},y) = j_*\ptf{x}+f^*y
\end{split}
\]
and the diagram is commutative. (See \cite[Chapter 17]{Fulton} for the operations on operational Chow groups.)
Since $\ptf{X}$ is a projective bundle, $\ptf{X}$ is Alexander if $X$ is Alexander.
By the 4-lemma it follows that the evaluation morphism $\ev_{\ptf{Y}}$ is an isomorphism.
\end{proof}

From the view point of intersection theory Alexander stacks behave very similar to smooth \DM stacks. In particular Chow rings have an intersection product and satisfy a projection formula.
Given an Alexander stack $Z$ and $\alpha \in \CH(Z)$, we let $T_{\alpha}$ be the unique operator with $T_{\alpha} \cap [Z] = \alpha$.
The intersection product on $\CH(Z)$ takes the form
\[
\alpha\cdot\beta = T_{\alpha}\cap \beta=\alpha\cap T_{\beta}= T_{\alpha}\cap T_{\beta}\cap [Z].
\]
For any proper, representable morphism $f:X \to Z$ there is a  projection formula
\begin{equation}\label{eq:projection}
f_*(f^*\beta\cdot \alpha) = \beta \cdot f_*\alpha.
\end{equation}
induced by the projection formula on operational Chow groups.

For any representable morphism  $f:X\to Y$ we have a pullback on operational Chow groups, on the hand if $f$ is additionally a local complete intersection morphism we also have a pullback on Chow groups.

\begin{lemma}Let $f:X\to Y$ be a representable  local complete intersection morphism of Alexander stacks. Then
\[
f^*T_{\alpha} = T_{f^*\alpha}.
\]
\end{lemma}

\begin{proof}
Since $f$ is a local complete intersection morphism, we can decompose it into a regular embedding and a smooth morphism. In particular we have $f^*[Y] = [X]$ by \cite[Lemma 5.5]{VistoliIntersection}.
Thus
\[
f^*T_{\alpha} \cap [Y] = f^*T_{\alpha}\cap f^*[X] = f^*(T_{\alpha}\cap [X])= f^*\alpha.
\]
Since $X$ is an Alexander stack, we conclude $f^*T_{\alpha} = T_{f^*\alpha}$.
\end{proof}

\begin{convention}
We assume for the rest of the appendix that all \DM stacks are Alexander stacks and from now on we identify $\CHop$ and $\CH$ via the evaluation map.
\end{convention}

\subsection*{Tautological rings}
Suppose $Y$ is an Alexander stack and $X\subseteq Y$ a closed substack.
Recall that $X$ is called regularly embedded if etale locally $Y$ is defined by a regular sequence. It follows that the normal sheaf $\nb{X}{Y}$ is a vector bundle of rank $\codim_Y X$.

Given a subring $\Rtr(Z)\subseteq \CH(Z)$ and a regular embedding $i:W\hookrightarrow Z$ we define $R^*(W)=i^*R^*(Z)$.
Additionally, let $\pi:\ptf{Z}\to Z$ be the blowup of $Z$ along $W$. We define $\Rtr(\ptf{Z})= \pi^*\Rtr(Z)[E]$, where $E$ denotes the exceptional divisor.

For a substack $\ptf{Y}$  of $\ptf{Z}$ we then also let $\Rtr(\ptf{Y})$ be the pullback of $\Rtr(\ptf{Z})$.

Thus for every regularly embedded subscheme of a blowup of $Z$ along a regularly embedded substack we have defined a ring $R^*(Y)$.
We call the assignment $Y\mapsto R^*(Y)$ a {\em tautological ring}.

\begin{remark}
Suppose $R^*: Y\subseteq Z$ is regularly embedded and let $\ptf{Z}=\Bl_W Z, \ptf{Y} = \Bl_{W\cap Y} Y$.  Then the inclusion $j:\ptf{Y}\to\ptf{Z}$ is a regular embedding (see \Cref{sec:equationproper}) and we have a commutative diagram 
\[
\begin{tikzcd}
\ptf{Y}\ar[d,swap, "g"]\ar[r,"j"] & \ptf{Z}\ar[d,"f"]\\
Y\ar[r,swap,"i",] & Z
\end{tikzcd}
\]
Furthermore, 
\[
R^*(\ptf{Y}) = j^*R^*(\ptf{Z}) = g^*(R^*(Y))[E_Y],
\]
where $E$ is the exceptional divisor in $\ptf{Z}$ and $E_Y=j^*E$ the exceptional divisor in $\ptf{Y}$.
In particular there is no ambiguity in defining $R^*$.
\end{remark}

\subsection*{Algebraically clean and transverse intersection}
\begin{definition}\label{appendix-transverse}
We say two substacks  $W,Y\subseteq Z$ intersect {\em algebraically cleanly} inside $Z$ if every morphism in the Cartesian diagram
\[
\begin{tikzcd}
X =W\times_Z Y\ar[r] \ar[d]& Y \ar[d]\\
W \ar[r] & Z
\end{tikzcd}
\]
is a regular embedding.
If furthermore
\[
\codim_Z(X) = \codim_Z(W) + \codim_Z(Y),
\]
we say the intersection is {\em algebraically transverse}.

\end{definition}

\begin{remark}
The analogy to clean and transverse intersection of smooth varieties is as follows.
If $X$ and $W$ have algebraically clean intersection, then etale-locally we can find
regular sequences such that the ideal sheaves $\calI(W)$ and $\calI(Y)$ are generated by 
\[
\calI(W) =  (a_1,\ldots,a_n),\quad
\calI(Y) = (a_1,\ldots,a_m,b_1,\ldots,b_k), m\leq n.
\]
Here we allow $m=0$, in which case the intersection is algebraically transverse.
\end{remark}

\subsection*{Total Chern classes}
\begin{definition} \label{sec:chp}

Suppose $i:V\hookrightarrow Z$ is  regularly embedded of codimension $k$.
We say that $V$ {\em has a \CHP} if there exists a polynomial \[
\chp[(t)]{V}{Z}=\cinbe[0]{V}{Z}+ \cinbe[1]{V}{Z}+\ldots+ \cinbe[k-1]{V}{Z}+\cinbe[k]{V}{Z}\in\CH(Z)
\] such that
\begin{itemize}
\item $\cinbe[0]{V}{Z}\cap[Z]=[Z],\, \cinbe[k]{V}{Z}\cap [Z]=[V]$.
\item
$ i^*\cinbe{V}{Z}=\cinb{V}{Z}$ for all $i$.
\end{itemize}
\end{definition}

Note that \CHPs are  unique only up to elements in  $\ker(i^*:\CH(Z)\to\CH(X))$.
We stress that our convention differs from \cite{FultonMacPherson} where instead a polynomial is used and also the indexing is different.

\begin{proposition} The \CHP has the following properties.
\begin{enumerate}
\item If $D\subseteq Z$ is a divisor, then $\chp{D}{Z}= 1+[D]$.
%\item Suppose $X\subseteq Y\subseteq Z$ are regularly embedded. If $X$ and $Y$ have a Chern polynomial in $Z$, then $X$ has a Chern polynomial in $Y$. Furthermore, if $[X]\in\CH(Y)$ is a pullback from $Z$, then $\chp{X}{Y}$ is a pullback from $Z$.

\item Suppose $X$ and $Y$ intersect algebraically transversely inside $Z$ and $X$ has a \CHP  $\chp{X}{Z}$ in $Z$. Then 
\[
\chp{X\cap Y}{Y} = \chp{X}{Z},
\]
 as well as 
\[
\chp{X\cap Y}{Z}= \chp{X}{Z}\cdot\chp{Y}{Z},
\]
omitting evident pullbacks.
\end{enumerate}
\end{proposition}

\begin{proof}
\cite[Lemma 5.1]{FultonMacPherson}
\end{proof}

Our main goal is to determine whether the \CHP can be chosen to be contained in some subring $\Rtr(Z)\subseteq \CH(Z)$. For example $R^*$ could be the tautological ring of the moduli space of curves.

\begin{definition}
Let $\Rtr(Z)\subseteq \CH(Z)$ be a tautological subring and $V\subseteq Z$ regularly embedded.

We say that $V$ {\em has a \CHP in $Z$ defined over} $\Rtr$ if there exists a \CHPs $\chp{V}{Z}\in \Rtr(Z)$. 
\end{definition}

\begin{proposition} \label{prop:chern} Suppose $X\subseteq Y\subseteq Z$ are regularly embedded substacks.

\begin{enumerate}
\item If $\chp{X}{Y},\chp{Y}{Z}$ are defined over $\Rtr$, then
$\chp{X}{Z}$ is defined over $\Rtr$.
\item If $X = Y \cap W \subseteq Z$ is an algebraically transverse intersection for some substack $W\subseteq Z$ and $\chp{W}{Z}\in R^*(Z)$, then $\chp{X}{Y}= i^*\chp{W}{Z}\in R^*(Y)$, where $i:Y\hookrightarrow Z$ is the inclusion.
\end{enumerate}

\end{proposition}

\subsection{Intersection theory of blowups}
Let $W\subseteq Z$ be a regular embedding of \DM stacks and consider the blowup $\ptf{Z}$ of $Z$ in $W$ with exceptional divisor $\ptf{W}$.
\[
\begin{tikzcd}
\ptf{W} \ar[r,"j"]\ar[d,"\pi'",swap] & \ptf{Z}\ar[d,"\pi"]\\
W\ar[r,"i", swap] & Z
\end{tikzcd}
\]
Let $s(W,Z)=\sum_i s_i(\nb{W}{Z})$ be the total Segre class of the normal bundle $\nb{W}{Z}$.
By the birational invariance of Segre classes we
 have
\[
s(W,Z)=\pi_{\ptf{W},*}s(\ptf{W},\ptf{Z})=\sum_{i\geq 0}\pi'_{*}(\calO(-1)^i\cap [\ptf{W}])=
\sum_{i\geq 0}s_{i-r+1}(\calN_{W/Z})
%j^*\sum_{i\geq 1}(-1)^{i+1}[\ptf{W}]^i.
\]
If $W$ has a \CHP in $Z$, then there exists classes $\cl{s}_i(\nb{W}{Z})\in\CH(Z)$ which restrict to $s_i(\nb{W}{Z})$ on $X$. We can then write

\begin{equation}\label{eq:pfE}
\pi_*[\ptf{W}]^i=(-1)^{i-1}\overline{s}_{i-\codim(W)}(\calN_{W/Z})\cap [W] \text{ for } i\geq 1.
\end{equation}

We have the following useful consequence.
\begin{proposition}
Let $R^*(Z)\subseteq\CH(Z)$ be a tautological ring
and $W\subseteq Z$ regularly embedded.
Furthermore, suppose $\chp{W}{Z}\in \Rtr(Z)$.
Consider the blowup $\pi:\ptf{Z}\to Z$  of $Z$ in $W$. Then $\pi_*$ preserves $R^*$, i.e.,
\[
\pi_*(R^*(\ptf{Z})) = R^*(Z).
\]
\end{proposition}

\begin{proof}

From $\pi_*\pi^*=\id$ it follows that $\pi_*(R^*(\ptf{Z}))\supseteq R^*(Z)$. On the other hand we  have by definition that $R^*(\ptf{Z})[\ptf{W}]$, thus it suffices to show that $\pi_*[\ptf{W}^i]\in R^*(Z)$, which is the content of \cref{eq:pfE}.
\end{proof}

\subsection{Local equations for proper transforms}
\label{sec:equationproper}Suppose $W, V\subseteq Z$ intersect algebraically cleanly in $Z$. Let $\ptf{Z}=\Bl_{W} Z$ and $X= W\cap V$.
Locally we can find a regular sequence $x_1,\ldots,x_n,y_1,\ldots,y_m$ such that the ideals of $X$ and $V$ are
$\calI(W)=(x_1,\ldots,x_n),\, \calI(V)=(x_1,\ldots,x_d, y_1,\ldots, y_m)$ with $d\leq n$.

Furthermore, the blowup $\ptf{Z}=\Bl_X Z$ is locally defined in  $Z\times\PP^{n-1}$ by the equations 
\[
\{ x_iT_j=x_jT_i \text{ for } 1\leq i<j\leq n\}.
\]
The proper transform $\ptf{V}$ of $V$ can be identified with $\Bl_{W\cap V} V$.
In the affine chart where $T_i\neq 0$, the equation for the exceptional divisor $\ptf{W}$ is $z_i=0$.
Local equations for the proper transform $\widetilde{V}$ in  $Z \times \PP^{n-1}$ are given by
\begin{equation} \label{prop:equationproper}
\calI(\widetilde{V})=(T_1,\ldots,T_d, y_1,\ldots,y_m),
\end{equation}
where $T_1,\ldots,T_d$ are projective coordinates.
In particular $\ptf{V}\subset\ptf{Z}$ is regularly embedded. (For details see \cite[Lemma 4.1]{Aluffi}).

\subsection{Total Chern classes of proper transforms}
Our next goal is the computation of classes and normal bundles for proper transforms along a blowup.
In the case $W\subseteq V\subseteq Z$ a formula is found in \cite{Fulton}, and was generalized to the general situation in \cite{Aluffi}.

The main result is that the class  of the proper transform $\ptf{V}$ as well as its normal bundle can be computed from the classes of  $W,V,V\cap W$ as well as their normal bundles.

\begin{corollary}\label{cor:3outof4}
Suppose $V$ and $W$ intersect algebraically cleanly in $Z$. Let $\ptf{Z}=\Bl_W Z$ and $\ptf{V}= \Bl_{X} V, X= V\cap W$.

Furthermore assume that 
\[
\chp{X}{V},\, \chp{X}{W},\,\chp{W}{Z}
\]
are defined over $R^*$.

 Then $\chp{V}{Z}\in R^*(Z)$ if and only if $\chp{\ptf{V}}{\ptf{Z}}\in R^*(\ptf{Z})$.

\end{corollary}

In fact we we will prove an explicit formula for the \CHP $\chp{\ptf{V}}{\ptf{Z}}$  from which \Cref{cor:3outof4} follows.

We start by setting up the notation for Aluffi's formula \cite{Aluffi}.
As before we consider the blowup $\ptf{Z}$ of $Z$ in $W$ with exceptional divisor $\ptf{W}$.
Furthermore $V$ and $W$ intersect cleanly in $Z$ and we set $X= V \cap W$. The exceptional divisor is $\ptf{W}$ and $\ptf{V}$ is the proper transform of $V$ is $\ptf{V}$. The situation is summarized in the following diagram.

\setlength{\perspective}{8pt}
\[\begin{tikzcd}[row sep={45,between origins}, column sep={40,between origins}]
      &[-\perspective] \ptf{X} \ar["\ptf{i}"]{rr}\ar["\ptf{j}" near end,swap]{dd}\ar["\pi_X",swap]{dl} &[\perspective] &[-\perspective] \ptf{V} \ar["\ptf{k}"]{dd}\ar["\pi_V", swap]{dl} \\[-\perspective]
    X \ar[crossing over,"i"]{rr} \ar["j", swap]{dd} & & V \\[\perspective]
      & \ptf{W}  \ar["\ptf{l}"]{rr} \ar["\pi_W"]{dl} & &  \ptf{Z} \ar["\pi_Z"]{dl} \\[-\perspective]
   W \ar["l",swap]{rr} && Z \ar[from=uu,crossing over,"k" near start]
\end{tikzcd}\]

We define the classes
\[
\begin{split}
C&= \sum_m C_m:=j^*c(\nb{W}{Z})c(\nb{X}{V})^{-1}\in \CH(X),\\
N&=\sum_n N_n:=c(\nb{X}{W})\in \CH(X)
\end{split}
\]

If we formally assume that $C$ and $N$ are the pullbacks of the total Chern class of  vector bundles $\mathbb{C},\mathbb{N}$ on $V$, then we can expand 
\[
 c(\mathbb{N})\cdot c(\mathbb{C}\otimes \calO(-\ptf{X}))= c(\mathbb{N})\cdot c(\mathbb{C}) +  Q(c_m(\mathbb{N}),c_n(\mathbb{C}),[\ptf{X}]) \cdot \ptf{X}\in\CH(V).
\]
for a polynomial $Q= \sum_k Q_k(c_m(\mathbb{N}),c_n(\mathbb{C}))[\ptf{X}]^k$.
While in general we cannot ensure the existence of a vector bundle, the polynomial $Q$ still appears in the following formula.

In the next proposition we omit some pullback maps or sometimes write $\alpha_{|V}$ for the pullback along an inclusion $V\subset V', \alpha\in \CH(V')$.

\begin{proposition}\label{prop:propertransform}Let $Z$ be an Alexander stack.  Assume $V,W\subseteq Z$ intersect cleanly and let $X= V\cap W$.
Furthermore, suppose that $
\chp{X}{W},\chp{X}{V}, \chp{W}{Z},
$
are all defined over $\Rtr$.

Let $\cl{C},\cl{N},\overline{X}\in \Rtr(Z)$ such that
\[
\cl{N}_{|X} = N, \cl{C}_{|X} = C\in\Rtr(X), \overline{X}_{|W} = X\in \Rtr(W).
\]

The class of the proper transform $\ptf{V}$ of $V$ has the following form
\[
[\ptf{V}] - \pi_Z^*[V] = - \sum_{k=1}^{d-e} \cl{C}_{d-e-k} \cdot \cl{X} \cap (-[\ptf{W}]^k)\in R^*(\ptf{Z})
\]
where $d=\codim_Z V,\, e= \codim_{W} X$.

For the normal bundle of the proper transform we have
\[
c(\nb{\ptf{V}}{\ptf{Z}}) - \pi_V^*c(\nb{V}{Z}) = \left[Q( \cl{N}_m, \cl{C}_n, [\ptf{W}]) \cap \ptf{W}\right]_{|\ptf{V}} \in  \Rtr(\ptf{V}).
\]
\end{proposition}

\begin{proof}
By \cite[Thm 6.7]{Fulton} we have
\begin{equation}\label{eq:properFulton}
[\ptf{V}] = \pi_Z^*[V] + \ptf{l}_*\left\{\dfrac{c(\pi_W^*\nb{W}{Z})}{c(\calO(\ptf{W}))} \cap \pi_W^* (c^{-1}(\nb{X}{V})\cap [X])\right\}^{d-1}
\end{equation}

By assumption we have
 $c(\nb{W}{Z})c^{-1}(\nb{X}{Y}) = j^*l^*\cl{C}, [X] = l^*\cl{X}$ for $\cl{C},\cl{X} \in \Rtr(Z)$.

We can now rewrite \cref{eq:properFulton} as
\[
\begin{split}
[\ptf{V}] &= \pi_Z^*[V]  +\ptf{l}_* \left\{ \dfrac{\ptf{l}^*\pi_Z^*\cl{C}}{ c(\calO(\ptf{W}))} \cap \ptf{l}^*\pi_Z^*\cl{X}\right\}^{d-1}\\
& =
 \pi_Z^*[V]  + \ptf{l}_*\sum_{k=1}^{d-e} \ptf{l}^*\cl{C}_{d-e-k}\cdot \ptf{l}^*\cl{X}\cap (-c_1(\calO(\ptf{W}))^{k-1}\\
 &=
\pi_Z^*[V]  -\sum_{k=1}^{d-e} \cl{C}_{d-e-k}\cdot \cl{X}\cap (-[\ptf{W}])^{k}\in\CH[d](\ptf{Z})\\
\end{split}
\]
where $e=\codim_W X$.
In particular the difference $[\ptf{V}]-\pi_Z^*[V]$ lies in  $R^*(\ptf{Z})$.

We now proceed with the normal bundle of the proper transform.
 By \cite[Thm. 4.2]{Aluffi} we can write
\[
\begin{split}
c(\nb{\ptf{V}}{\ptf{Z}})\cap\alpha -  \pi_V^*(\nb{V}{Z})\cap\alpha  &= \ptf{i}_*\left[Q( \pi_X^*c_{m}(\nb{X}{W}), \pi_X^*C_{n}, c_1(\calO_{\ptf{X}}(-1)) \cap \ptf{i}^* \alpha\right]\\=&
\sum_{k=0} \ptf{i}_*\left[Q_{k}(  \cl{N}_m, \cl{C}_{n}) \cdot c_1(\calO_{\ptf{X}}(-1))^k \cap \ptf{i}^*\alpha\right]\\
&=  \sum_{k=0} Q_{k}(\cl{N}_m,\cl{C}_{n}) \cdot [\ptf{X}]^{k+1} \cap \alpha \\
&=  \left[\sum_{k=0} \left(Q_{k}(\cl{N}_m,\cl{C}_{n}) \cdot [\ptf{W}]^{k+1}\right) \cap \alpha \right]_{|\ptf{V}}. 
\end{split}
\]

\end{proof}

\begin{proof}[Proof of \Cref{cor:3outof4}]
By \Cref{prop:propertransform}
the differences
\[
[\ptf{V}] - \pi_Z^*[V] ,\, c(\nb{\ptf{V}}{\ptf{Z}}) - \pi_V^*\nb{V}{Z}
\]
are both defined over $R^*$ and can both be expressed solely in terms of the \CHPs $\chp{X}{V},\chp{X}{W},\chp{W}{Z}$ and the exceptional divisor $\ptf{W}$.
In particular, if $\chp{V}{Z}$ is defined over $\Rtr$ the same is true for $\chp{\ptf{V}}{\ptf{Z}}$.

For the second claim we recall that $\pi_{Z,*}\pi_Z^*=\id, \pi_{V,*}\pi_V^*=\id$
and furthermore that
$\pi_{Z,*}[\ptf{W}]^k$ can be computed in terms of the Segre classes $s(W,Z)= c^{-1}(\nb{W}{Z})\cap [W]$. Similarly
$\pi_{V,*}[\ptf{X}]^k$ can be computed in terms of the Segre classes  $s(X,V)= c^{-1}(\nb{X}{V})\cap [X]$. All these classes can be read off from the \CHPs $\chp{W}{Z}$ and $\chp{X}{V}$.
\end{proof}

\subsection{Systems of regular embeddings}
So far we have discussed \CHPs of proper transforms for a single blowup, we now generalize the computation of for a sequence of blow ups.
Roughly speaking, if we blowup $Z$ iteratively in a sequence of subspaces $W_1,\ldots,W_n$ and can compute the \CHPs of $W_i$ as well as the \CHP of all finite intersections of different $W_j$, then we can compute the \CHP of the proper transform of some subspace $V$, provided we also know the \CHP of all intersections $V\cap W$ where $W$ is any finite intersection of varieties $W_i$.

\begin{definition}
A  collection $S$ of substacks of $Z$ is called a {\em system of regular embeddings} in $Z$ if
\begin{itemize}
\item every stack $Y\in S$ is an Alexander stack,
\item $Z\in S$ and $S$ is closed under finite intersections
\item For any $A,B,C \in S$ with $A, B\subseteq C$ the intersection of $A$ and $B$ inside $C$ is clean.
\end{itemize}

A {\em building set} $B=(W_1,\ldots,W_n), W_i\in S$ is an ordered list of subspaces such that $W_i$ and $W_j$ are not contained in each other for $i\neq j$.
\end{definition}

Starting with a system $S$ of regular embeddings in $Z$ and a building set $B\subseteq S$ we construct an iterative blowup $\ptf{Z}$ as follows.
We start by setting \[
\ptf{Z}_0 = Z, S_0=S, B_0 =B, \ptf{Z}_0:=Z_0, \ptf{W}^{(0)}=W \text{ for all } W\in S_0.
\] 

Suppose we have already constructed  $\ptf{Z}_i, \ptf{W}^{(i)},S_i,B_i$.
We then set $\ptf{Z}_{i+1} = \Bl_{W_{i+1}^{(i)}} \ptf{Z}_i$.
For any subspace $W^{(i)}\in S_i$ we let  $W^{(i+1)}$ be the proper transform of $W^{(i)}$ if $W^{(i)}$ is not contained in $W_{i+1}^{(i)}$ and the preimage under the blowup otherwise.
We call $W^{(i+1)}$ the $i$-th iterated transform of $W$.

Afterwards we define
\[
S_{i+1}:=\{ \ptf{W}^{(i+1)}\,:\, \ptf{W}^{(i)}\in S_i, W\not\subseteq W_{i+1}\},\quad B_{i+1} = (\ptf{W}^{(i+1)}_{i+2},\ldots,\ptf{W}^{(i+1)}_n).
\]

We call the resulting space $\ptf{Z}:=\ptf{Z}_n$ the iterated blowup associated to $B$.

\begin{proposition}\label{prop:regular-system-blowup}
Suppose $S$ is system of regular embeddings in $Z$ with building set $B$.

For any $i$ the equality 
\[
\ptf{A}_1^{(i)} \cap \ptf{A}_2^{(i)}= \ptf{A_1\cap A_2}^{(i)} 
\]
holds for all $\ptf{A}_1, \ptf{A}_2 \in S$.
If $A_1$ and $A_2$ intersect algebraically transversely, then the same is true for $\ptf{A}_1^{(i)},\ptf{A}^{(i)}_2$.

Furthermore, for every $i$ the collection of subspaces
$S_i$ is a system of regular embeddings in $\ptf{Z_i}$ and $B_i$ is a  building set.

\end{proposition}

\begin{proof}
The proof is by induction on $i$. The case $i=0$ is true by definition.

For the rest of the proof we set $X_1= \ptf{A}_1^{(i-1)},\, X_2= \ptf{A}_2^{(i-1)},\, X_3= W_i^{(i-1)}$.

If we have already proven that $S_{i-1}$ is a system of regular embeddings, then it follows that  $S_i$ is also a system of regular embeddings  since we can explicitly compute regular sequences defining the ideal of the proper transform (see \Cref{prop:equationproper}) and the preimage.

We can reduce the local case where $Z=\Spec R$ and $ X_i, i=1,2,3$ are affine.
Since outside $X_3$ the blowup is an isomorphism, it suffices to check the equality near $x\in X_3$.
Let $\calI_i$ be the ideal defining $X_i$ in a neighborhood of $x$.
We claim that there exists tuples  of regular sequences $a, b_{12},b_{13},b_{23},c_1,c_2,c_3$ such that
the concatenation $(a, b_{12},b_{13},b_{23},c_1,c_2,c_3)$ is a regular sequence in $R$ and 
\begin{gather*}
\calI_1 = (a,b_{12},b_{13},c_1), \quad
\calI_2 = (a,b_{12},b_{23},c_2),\quad
\calI_3 = (a,b_{13},b_{23},c_3),\quad \\
\calI_1+\calI_2 = (a,b_{12},b_{13},b_{23},c_1,c_2),\\
\calI_1+\calI_2+\calI_3 = (a,b_{12},b_{13},b_{23},c_1,c_2,c_3).
\end{gather*}

Let $x\in Z$. To construct such a tuple of  regular sequence in a neighborhood of $x$ we can lift  a suitable  basis for the $\kappa(x)$-vector space $(\calI_1+\calI_2+\calI_3)/\mathfrak{m}_x(\calI_1+\calI_2+\calI_3)$.
By \cite[\href{https://stacks.math.columbia.edu/tag/067N}{Tag 067N}]{stacks-project} the lift is a regular sequence.

Next we can compute the ideal sheaves of the proper transform following \Cref{sec:equationproper}.
For every element $\alpha$ $\calI_3$  we introduce a corresponding projective coordinate $\tilde{\alpha}$.
Let $\ptf{a},\ptf{b}_{13},\ptf{b}_{23},\ptf{c}_3$ be the corresponding tuples of projective coordinates.
As before the sequence $(\ptf{a},b_{12},\ptf{b}_{13},\ptf{b}_{23},c_1,c_2,\ptf{c}_3)$ is a regular sequence on the blowup $\ptf{Z}_i$ and for the ideal sheaves $\ptf{\calI}_i$ of the proper transforms we find
\begin{gather*}
\ptf{\calI}_1 = (\ptf{a},b_{12},\ptf{b}_{13},c_1), \quad
\ptf{\calI}_2 = (\ptf{a},b_{12},\ptf{b}_{23},c_2),\\
\ptf{\calI}_1+\ptf{\calI}_2 = (\ptf{a},b_{12},\ptf{b}_{13},\ptf{b}_{23},c_1,c_2).
\end{gather*}
Therefore $\ptf{\calI}_1+\ptf{\calI}_2 = \ptf{\calI_1+\calI_2}$.

\end{proof}

\begin{definition}
Let $S$ a system of regular embeddings and $B$ a building set.
We say that $S$  {\em has enough \CHPs} in $R^*$ if  for
$V,W \in S$ with $V \subseteq W$ the \CHP  $\chp{V}{W}$ exists and is defined over $\Rtr$.

\end{definition}

\begin{theorem}\label{prop:arrangement}
Let $S$ be a system of regular embeddings that has enough \CHPs and $B$ a building set with iterated blowup $\ptf{Z}\to Z$.

Let $A\subseteq Z$, not necessarily contained in $S$.
We assume that for any $W\in S$  the intersection $W\cap A$ is clean and that all \CHPs
\[
\chp{A\cap W}{A},\chp{A\cap W}{W}, \chp{\ptf{A}}{\ptf{Z}}
\]
are defined over $R^*$.

Then $A$ has a \CHP defined  over $R^*(Z)$ and it can be computed explicitly.
\end{theorem}

\begin{proof}
By induction and \Cref{prop:regular-system-blowup} we can reduce to the case of a single blowup, i.e., $B=\{W_1\}$.
Since all \CHPs
\[
\chp{A\cap W_1}{A},\chp{A\cap W_1}{W_1},\chp{W_1}{Z}, \chp{\ptf{A}}{\ptf{Z}}
\]
are defined over $\Rtr$, the claim follows from \Cref{cor:3outof4}.
\end{proof}


\clearpage
\begin{thebibliography}{BCGGM19}



\bibitem[AP19]{APDRToric}
A. Abreu and M.Pacini,
\newblock {\em The resolution of the universal Abel map via tropical geometry and applications},
\newblock Preprint \href{https://arxiv.org/abs/1903.08569}{arxiv:1903.08569}.



\bibitem[Alu10]{Aluffi}
P. Aluffi \newblock {\em Chern classes of blow-ups.}\newblock  Mathematical Proceedings of the Cambridge Philosophical Society, 148(2), 227-242. (2010)


\bibitem[BCGGM18]{BCGGMgrc}
Matt Bainbridge, Dawei Chen, Quentin Gendron, Samuel Grushevsky, and Martin
  M{\"o}ller.
\newblock {\em Compactification of strata of abelian differentials}.
\newblock Duke Math. J.  167.12 (2018), pp. 2347--2416.

\bibitem[BCGGM19a]{BCGGMk}
M. Bainbridge, D. Chen, Q. Gendron, S. Grushevsky, and M. M{\"o}ller.
\newblock {\em Strata of k-differentials}.
\newblock Algebr. Geom. 6.2 (2019), pp. 196–233.

\bibitem[BCGGM19b]{BCGGMsm}
M. Bainbridge, D. Chen, Q. Gendron, S. Grushevsky, and M. M{\"o}ller.
\newblock {\em The moduli space of multi-scale differentials}.
\newblock Preprint \href{https://arxiv.org/abs/1910.13492}{arXiv:1910.13492}, 2019.


\bibitem[BHPSS21]{universal-dr},
     Y. Bae, D. Holmes, R. Pandharipande, J. Schmitt, R. Schwarz. \newblock {\em Pixton's formula and Abel-Jacobi theory on the Picard stack}
     \newblock arXiv preprint href{https://arxiv.org/abs/2004.08676}{arXiv:2004.08676}, 2021
    



\bibitem[Ben22]{Fred}
F. Benirschke.
 \newblock {\em
The boundary of linear subvarieties in strata of differentials}
\newblock  J. Eur. Math. Soc. (2022).




\bibitem[Ben20]{FredDR}
F. Benirschke. \newblock {\em 
The closure of double ramification loci via strata of exact differentials. }
\hyperref{https://doi.org/10.48550/arxiv.2012.07703}{arxiv.2012.07703} (2020)

\bibitem[BDG20]{BDG}
F. Benirschke, B. Dozier and S. Grushevsky
\newblock {\em Equations of linear subvarieties of strata of differentials}
\newblock 
Geometry and Topology 26, pp.2773–2830, 2022.




\bibitem[BRZ21]{residueless}
A. Buryak, P. Rossi, D. Zvonkine. \newblock {\em Moduli spaces of residueless meromorphic differentials and the {KP} hierarchy.} \newblock Preprint \href{https://arxiv.org/abs/2110.01419}{arXiv:2110.01419}  2021

\bibitem[Che17]{ChenSurvey}
D. Chen.
\newblock  {\em Teichm{\"u}ller dynamics in the eyes of an algebraic geometer}.
\newblock Surveys on recent developments in algebraic geometry, Volume 95 ofProc. Sympos. Pure Math. (2017), pp. 171–197, Amer. Math.Soc., Providence, RI

 
\bibitem[CGHMJ22]{tale-moduli}  
D. Chen, S. Grushevsky, D. Holmes,M.  M\"{o}ller, J. Schmitt. \newblock {\em A tale of two moduli spaces: logarithmic and multi-scale differentials}. \newblock arXiv Preprint 
\href{https://arxiv.org/abs/2212.04704}{arXiv:2212.04704}, 2022.


\bibitem[CMS19]{chen-moeller-sauvaget}
D> Chen,  M. M\"{o}ller and A. Sauvaget,
\newblock {\em Masur-Veech volumes and intersection theory: the principal strata of quadratic differentials} 
\newblock Ariv Prepreint 
\href{https://arxiv.org/abs/1912.02267}{arXiv:1912.02267}, 2019.


\bibitem[CMSZ20]{CMSZ}
D. Chen, D, M. M{\"o}ller, M., A. Sauvaget, D. Zagier \newblock {\em  {M}asur–{V}eech volumes and intersection theory on moduli spaces of {A}belian differentials}. 
\newblock Invent. math. 222, pp. 283–373 (2020). 

\bibitem[CMS23]{euler-linear}
M. Costantini, M. M{\"o}ller, J. Schwab.
\newblock {\em Chern classes of linear submanifolds with application to spaces of k-differentials and ball quotients}, 
\newblock 
\href{https://arxiv.org/abs/2303.17929}{arXiv:2303.17929}, 2023.
 


\bibitem[CMZ20a]{CMZ}
M. Costantini, M. M{\"o}ller, J. Zachhuber.
\newblock {\em
The {C}hern classes and the {E}uler characteristic
of the moduli space of {A}belian differentials}
\newblock Preprint \href{https://arxiv.org/abs/2006.12803}{arXiv:2006.12803}, 2020

\bibitem[CMZ20b]{diffstrata}
M. Costantini, M. M{\"o}ller, and J. Zachhuber.
\newblock {\em  diffstrata – a Sage package for calculations in the tautological ring of the moduli space of {A}belian differentials}.
\newblock  2020. arXiv: 2006.12815.



\bibitem[DSZ21]{admcycles}
V. Delecroix, J. Schmitt, and J. van Zelm.\newblock{\em admcycles—a Sage pack- age for calculations in the tautological ring of the moduli space of stable curves}.\newblock In: J. Softw. Algebra Geom. 11.1 (2021), pp. 89–112.



\bibitem[EH87]{eisenbud-harris}
D. Eisenbud and J. Harris. \newblock {\em The {K}odaira dimension of the moduli space of curves of genus $\geq$ 23}. \newblock Inventiones Mathematicae, 90(2):359–387, (1987)

\bibitem[FP05]{faber-pandharipande}
C. Faber and  R. Pandharipande. \newblock {\em Relative maps and tautological classes.}
\newblock Journal of the European Mathematical Society, 7.1 (2005), pp.13-49.

\bibitem[FP15]{faber-pagani}
C. Faber and N. Pagani. \newblock {\em The class of the bielliptic locus in genus 3.} \newblock International Mathematics Research Notices, 2015(12):3943–3961 (2015).

\bibitem[FP18]{FarkasPandharipande}
Farkas, G. and Pandharipande, R.
\newblock {\em The moduli space of twisted canonical divisors.}
\newblock Journal of the Institute of Mathematics of Jussieu, 17.3 (2018), pp. 615-672.

\bibitem[Ful98]{Fulton}
W. Fulton \newblock {\em Intersection theory}
\newblock 
Springer-Verlag Berlin Heidelberg 1998,
 ISBN
978-0-387-98549-7, Published: 26 June 1992, Edition number 2


\bibitem[FM94]{FultonMacPherson}
W. Fulton and R. MacPherson. \newblock {\em A Compactification of Configuration Spaces}.\newblock Annals of Mathematics 139, no. 1 (1994), pp.183–225.

\bibitem[Gen15]{QuentinThesis}
Q. Gendron.
\newblock {\em The {D}eligne-{M}umford and the Incidence Variety Compactifications of the Strata of $\Omega\mathcal{M}_{g}$}.
\newblock  Annales de l’institut Fourier. 68(2015), 10.5802/aif.3187.


\bibitem[GP03]{graber-pandharipande}
T> Graber and R. Pandharipande. \newblock {\em Constructions of nontautological classes on moduli spaces of curves.} Michigan Math. J. 51 (2003), no. 1, pp. 93–109. 


\bibitem[GV05]{GV} Tom  Graber  and  Ravi  Vakil,
\newblock{ Relative  virtual  localization  and  vanishing  of tautological  classes  on  moduli  spaces  of  curves},
\newblock   Duke  Math. J.130.1(2005), pp. 1–37.


\bibitem[GZ14]{GZ}
S. Grushevsky and D. Zakharov.
\newblock {\em
 The zero section of the universal semiabelian variety and the double ramification cycle}.
 \newblock  Duke Math. J. 163.5 (2014), pp. 953--982.


\bibitem[HKP18]{HKPDR}
D. Holmes, J. Kass, and N. Pagani.
\newblock {\em  Extending the double ramification cycle using Jacobians}.
\newblock Eur. J. Math., 4.3 (2018), pp. 1087–1099.

\bibitem[Hol19]{HolmesDR}
D. Holmes,
\newblock {\em Extending  the  double  ramification  cycle by  resolving  the  {A}bel-{J}acobi  map}.
\newblock J.  Inst.  Math.  Jussieu, 1-29, 2019.

\bibitem[JRPZ17]{DRFormula} F.  Janda,  R.  Pandharipande,  A.  Pixton,  and  D.  Zvonkine,
\newblock {\em Double  ramification  cycles  on  the  moduli  spaces  of  curves},
\newblock Publications math{\'e}matiques de l’{IH{\'E}S} 125.1(2017), pp. 221–266


\bibitem[KL22]{Klingler-Lerer}
B. Klingler and L. Lerer,
\newblock {\em {A}belian differentials and their periods: the bialgebraic point of view}. \newblock
arXiv preprint
  
  \hyperref{https://arxiv.org/abs/2202.06031}{arXiv:2202.06031} (2022)

\bibitem[Li01]{Li1} J. Li.
\newblock {\em Stable morphisms to singular schemes and relative stable morphisms},
\newblock J. Differential Geom. 57.3(2001), 509–578.

\bibitem[Li02]{Li2} J. Li.
\newblock{ A degeneration formula of GW-invariants},
\newblock J. Differential Geom. 60.2(2002), no. 2, 199–293.

\bibitem[Lia21]{Lian}
C. Lian. \newblock {\em 
The {H}-tautological ring}
\newblock Selecta Math. (N.S.) 27 (2021), Article Nr. 96.


\bibitem[MPS23]{log-chow}
S. Molcho, R. Pandharipande and J. Schmitt.\newblock {\em  The Hodge bundle, the universal 0-section, and the log Chow ring of the moduli space of curves.} \newblock Compositio Mathematica, 159(2), pp. 306-354, 2023.

\bibitem[Mul22]{mullane}
S. Mullane. \newblock {\em 
Strata of differentials of the second kind, positivity, and irreducibility of certain {H}urwitz spaces.} \newblock
Annales de l'Institut Fourier, Volume 72 (2022) no. 4, pp. 1379-1416

\bibitem[Mum83]{mumford-enumerative}
D. Mumford. \newblock {\em Towards an enumerative geometry of the moduli space of curves}. \newblock Arithmetic and geometry,  1983, pp. 271-328.

\bibitem[PP06]{hurwitz-existence}
E. Pervova and C. Petronio. \newblock {\em On the existence of branched coverings between surfaces with prescribed branch data, I.} \newblock Algebraic \& Geometric Topology 6.4, pp. 1957-1985, 2006.


\bibitem[Sau17]{sauvaget-thesis}
A. Sauvaget. \newblock {\em Th{\'e}orie de l'intersection sur les espaces de diff{\'e}rentielles holomorphes et m{\'e}romorphes}. Thesis

\bibitem[Sau19]{sauvaget-cohomology}
A. Sauvaget. \newblock {\em Cohomology classes of strata of abelian differentials.}\newblock Geometry and Topology, 23(3):1085–1171 (2019). 


\bibitem[SvanZ20]{schmitt-vanzelm}
J. Schmitt and J. van Zelm. \newblock{\em Intersections of loci of admissible covers with tautological classes}. \newblock Sel. Math. New Ser. 26, 79 (2020)

\bibitem[Sta18]{stacks-project}
  The {Stacks Project Authors}. \newblock {\em {Stacks}}.
     \newblock \textit{Stacks Project},
 \url{https://stacks.math.columbia.edu},
 (2023)

\bibitem[Vis89a]{VistoliIntersection}
A. Vistoli \newblock {\em
Intersection theory on algebraic stacks and on their moduli spaces} \newblock 
Inventiones mathematicae
90.3 (1989), pp.
613--670



\bibitem[Vis89b]{VistoliAlexander}
A. Vistoli. \newblock {\em Alexander duality in intersection theory.} Compositio Mathematica 70.3 (1989), pp. 199-225.

\end{thebibliography}
\end{document}